\documentclass[11pt]{article}
\usepackage{latexsym,amsfonts,amssymb,amsmath,amsthm}
\usepackage[pagewise]{lineno}
\usepackage{color,comment}
\usepackage{graphicx}

\begin{document}

\newtheorem{thm}{Theorem}[section]
\newtheorem{cor}{Corollary}[section]
\newtheorem{lem}{Lemma}[section]
\newtheorem{prop}{Proposition}[section]
\newtheorem{defn}{Definition}[section]
\newtheorem{rk}{Remark}[section]
\newtheorem{nota}{Notation}[section]
\newtheorem{Ex}{Example}[section]
\def\nm{\noalign{\medskip}}

\numberwithin{equation}{section}

\newcommand{\ds}{\displaystyle}
\newcommand{\pf}{\medskip \noindent {\sl Proof}. ~ }
\newcommand{\p}{\partial}
\renewcommand{\a}{\alpha}
\newcommand{\z}{\zeta}
\newcommand{\pd}[2]{\frac {\p #1}{\p #2}}
\newcommand{\norm}[1]{\left\| #1 \right \|}
\newcommand{\dbar}{\overline \p}
\newcommand{\eqnref}[1]{(\ref {#1})}
\newcommand{\na}{\nabla}
\newcommand{\Om}{\Omega}
\newcommand{\ep}{\epsilon}
\newcommand{\tmu}{\widetilde \epsilon}
\newcommand{\vep}{\varepsilon}
\newcommand{\tlambda}{\widetilde \lambda}
\newcommand{\tnu}{\widetilde \nu}
\newcommand{\vp}{\varphi}
\newcommand{\RR}{\mathbb{R}}
\newcommand{\CC}{\mathbb{C}}
\newcommand{\NN}{\mathbb{N}}
\renewcommand{\div}{\mbox{div}~}
\newcommand{\bu}{{\bf u}}
\newcommand{\la}{\langle}
\newcommand{\ra}{\rangle}
\newcommand{\Scal}{\mathcal{S}}
\newcommand{\Lcal}{\mathcal{L}}
\newcommand{\Kcal}{\mathcal{K}}
\newcommand{\Dcal}{\mathcal{D}}
\newcommand{\tScal}{\widetilde{\mathcal{S}}}
\newcommand{\tKcal}{\widetilde{\mathcal{K}}}
\newcommand{\Pcal}{\mathcal{P}}
\newcommand{\Qcal}{\mathcal{Q}}
\newcommand{\id}{\mbox{Id}}
\newcommand{\stint}{\int_{-T}^T{\int_0^1}}

\newcommand{\be}{\begin{equation}}
\newcommand{\ee}{\end{equation}}

\newcommand{\rd}{{\mathbb R^d}}
\newcommand{\rr}{{\mathbb R}}
\newcommand{\alert}[1]{\fbox{#1}}
\newcommand{\eqd}{\sim}
\def\R{{\mathbb R}}
\def\N{{\mathbb N}}
\def\Q{{\mathbb Q}}
\def\C{{\mathbb C}}
\def\ZZ{{\mathbb Z}}
\def\l{{\langle}}
\def\r{\rangle}
\def\t{\tau}
\def\k{\kappa}
\def\a{\alpha}
\def\la{\lambda}
\def\De{\Delta}
\def\de{\delta}
\def\ga{\gamma}
\def\Ga{\Gamma}
\def\ep{\varepsilon}
\def\eps{\varepsilon}
\def\si{\sigma}
\def\Re {{\rm Re}\,}
\def\Im {{\rm Im}\,}
\def\E{{\mathbb E}}
\def\P{{\mathbb P}}
\def\Z{{\mathbb Z}}
\def\D{{\mathbb D}}
\def\p{\partial}
\newcommand{\ceil}[1]{\lceil{#1}\rceil}

\title{Pattern formation in a cell migration model with aggregation and diffusion}

\author{Lianzhang Bao\thanks{Department of Mathematics and Statistics,
Auburn University,  AL 36849, U. S. A. (lzb0059@auburn.edu).}}

\date{}

\maketitle

\begin{abstract}
In this paper, we study pattern formations in an aggregation and diffusion cell migration model with Dirichlet boundary condition. The formal continuum limit of the model is a nonlinear parabolic equation with a diffusivity which can become negative if the cell density is small and spatial oscillations and aggregation occur in the numerical simulations. In the classical diffusion migration model with positive diffusivity and non-birth term, species will vanish eventually with Dirichlet boundary. However, because of the aggregation mechanism under small cell density, the total species density is conservative in the discrete aggregation diffusion model. Also, the discrete system converges to a unique positive steady-state with the initial density lying in the diffusion domain. Furthermore, the aggregation mechanism in the model induces rich asymptotic dynamical behaviors or patterns even with 5 discrete space points which gives a theoretical explanation that the interaction between aggregation and diffusion induces patterns in biology. In the corresponding continuous backward forward parabolic equation, the existence of the solution, maximum principle, the asymptotic behavior of the solution is also investigated.
\end{abstract}



\pagestyle{myheadings}
\thispagestyle{plain}
{}

\section{Introduction}
Many models have been proposed for spatial pattern formation in cell evolutions and analyzed for the standard case of zero-flux boundary conditions (see \cite{Baker2019afree}, \cite{Bao2020continuum}, \cite{Bubba2020from}, \cite{Issa2017dynamics}, \cite{Issa2019persistence}, \cite{Myerscough1998pattern}, \cite{Padron1998sobolev}, \cite{Turner2004from} and the references therein). However, relatively little attention has been paid to the role of of boundary conditions on the form of the final pattern.
The current paper is  to study the asymptotic dynamical behaviors the following aggregation diffusion lattice model with Dirichlet boundary condition:
\begin{equation}\label{lattice-eq}
 u_j^{t+\tau}=u_j^t+\frac{u_j^tu_{j-1}^t}{2}(u_j^t+u_{j-1}^t-1)(u_{j-1}^t-u_j^t)+\frac{u_j^tu_{j+1}^t}{2}(u_j^t+u_{j+1}^t-1)(u_{j+1}^t-u_j^t),
\end{equation}
where  $u^t_j = u(x_j,t),u_0^t=u_N^t =0,$ and $u^{t+\tau}_{j}=u(x_j,t+\tau)$ where $x_j=\frac{j}{N} (j=0,\cdots, N)$ and $\tau$ is the time period. Equation \eqref{lattice-eq} can be viewed as a discrete version of the following backward forward parabolic equation with Dirichlet boundary:
\begin{equation}\label{continu-eq}
 u_t= [ D(u)u_x]_x \quad (x,t)\in Q_T,
\end{equation}
where $D(u) = u^2(u-\frac{1}{2})$ and $Q_T:=[0,1]\times[0,T]$. Equation \eqref{continu-eq} is a special case of backward-forward parabolic equation and the background is from the pattern formation or attraction and repulsion phenomena in biology. Aggregative behaviour is an important factor influencing survival and reproduction of
animals. For instance, both theories and experiments suggest that gregarious
behaviour can increase an animal's chances of avoiding capture by a predator \cite{Turchin1989population}.
There is by now a vast literature devoted to modeling the migration of cell populations (see \cite{KA1}, \cite{KA2}, \cite{KC}, \cite{AG} ,\cite{Baker2019afree}, \cite{BaoShen1}, \cite{BaoShen2}, \cite{Bubba2020from}, \cite{Deroulers2009modeling}, \cite{HPO}, \cite{Linana1999spatially},\cite{M}, \cite{Skellam1952random} and the references therein).

Models of cell migration generally come in two complementary forms: stochastic, individual-based or deterministic, population-based. Simulations of stochastic models often require more computational work, but do include the randomness that is often prevalent in biological system. Population-based models, which usually involve systems of partial differential equations (PDEs), generally not only require less computational work to obtain a numerical solution, but one can also use many tools from PDE analysis to explore their behaviors \cite{Thompson2012modelling}.

There are several types of individual-based model. One of the most commonly used is the space-jump model, where each cell moves around in space on a lattice, jumping from its current compartment to a neighbouring one (see \cite{Bao2020continuum}, \cite{Bao2014traveling}, \cite{Bubba2020from}, \cite{Skellam1952random}, \cite{Turchin1989population} and the references therein). Other models include velocity-jump models, where each cell repeatedly jumps between different velocities. Individual-based models have found applications in ecology, pattern formation, wound healing, tumor growth and gastrulation and vasculogenesis in the early embryo, amongst many others.

Population-based models frequently involve the development of a reaction-diffusion equation (see \cite{Bao2020continuum},\cite{HPO}, \cite{Keller1970Initiation},\cite{Keller1971Model}, \cite{Turchin1989population}, \cite{Turner2004from} and the references therein). These are useful when the length scale we wish to investigate is much greater than the diameter of the individual elements composing it. These models have been found to be particularly useful in the study of pattern formation in nature, especially the phenomenon of ``diffusion-driven instability" (see \cite{KA1}, \cite{KC}, \cite{AG}, \cite{Bao2020continuum}, \cite{Bao2014traveling},\cite{HPO}, and the references therein).
Keller-Segel model is commonly used population-based models for cell migration and aggregation due to chemotaxis \cite{Keller1970Initiation}, \cite{Keller1971Model}. Various biological phenomena, including the movement of  {\it Escherichia coli} \cite{Keller1971Model} and {\it Dictyostelium discoideum} \cite{Keller1970Initiation}, have been represented using the Keller-Segel model. A simplified version of the model involves the distribution $u$ of the density of the slime mold \textit{Dyctyostelum discoideum} and the concentration $v$ of a certain chemoattractant satisfying the following system of partial differential equations
\begin{equation}\label{KS}
\begin{cases}
 u_t = \nabla\cdot(\nabla u - \chi (u,v)\nabla v), \quad x\in\Omega
 \\
 \epsilon v_t =d\Delta v + F(u,v),\quad x\in\Omega
\end{cases}
\end{equation}
complemented with certain boundary condition on $\partial \Omega$ if $\Omega$ is bounded, where $\Omega \subset \mathbb{R}^N$ is an open domain, $\epsilon \geq 0$ is a non-negative constant linked to the speed of diffusion of the chemical, $\chi$ represents the sensitivity with respect to chemotaxis, $d$ is the diffusion coefficient of chemoattractant $v$, and the function $F$ models the growth of the chemoattractant.

In the special case when chemoattractant $v$ diffuses faster than slime mold $u$ and chemoattractant $v$ will reach the steady state and $\epsilon =0$. Furthermore, when let $d=0, F(u,v)=u-v, \chi(u,v) = 1- u^2(u-\frac{1}{2})$, Equation \eqref{continu-eq} is derived from Keller-Segel model \eqref{KS}. From this point of view, backward-forward parabolic equation \eqref{continu-eq} is a special Keller-Segel model.

When both individual-based and population-based models are studied simultaneously for the same cell's behaviors, similar results can be expected for both models at the length scales where their ranges of applicability overlap. However, if there are significant differences in the both models describing the same cell migration behaviors, there may exist important relations between the individual and population levels and which are neglected in the analysis.

Because the backward-forward parabolic equation \eqref{continu-eq} is ill-posed (see \cite{Bao2020continuum}, \cite{Helmers2013interface}, \cite{Helmers2018hysteresis}, \cite{Smarrazzo2012degenerate}), the objective of the current paper is mainly to investigate the asymptotic dynamical behaviors of Equation \eqref{lattice-eq} with Dirichlet boundary condition. More specifically, the bounded-ness and asymptotic behaviors of the lattice model solution in the diffusion or forward regions will be explored. For the general initial solution, special 5 points lattice model with Dirichlet boundary condition will have many asymptotic dynamical behaviors which gives a theoretical explanation of pattern formation in biology. The biological incentive to study these questions lies in that the species tends to aggregate when they are under surviving risk and their habitat (aggregation/backward region) will shrink as time continues and the total density is conserved which is the key difference compared to the pure diffusion or forward parabolic equation where its domain only has the spreading property and the total density will extinct eventually with Dirichlet boundary condition.

\par The rest of the paper is organized as follows. In section 2, we derive the aggregation diffusion lattice cell migration model via a biased random walk, then combined with ``the diffusion approximation'' process, the corresponding continuous backward-forward parabolic equation is derived. In section 3, we explore the properties of the aggregation diffusion lattice model \eqref{lattice-eq} with Dirichlet boundary condition, more precisely, the bounded-ness, asymptotic behaviors of the lattice solution in the diffusion or forward region, asymptotic dynamical behaviors of solution with general initial solution in 5 discrete points case will be investigated. In section 4, we study the existence and nonexistence, maximum principle, and asymptotic behaviors of solution of Equation \eqref{continu-eq} under different specific conditions.

\section{The derivation of the aggregation diffusion equation}
By the need for survival, mating or to overcome the hostile environment, the population have the tendency of aggregation when the population density is small and diffusion otherwise. For simplicity, we consider one species living in a one-dimensional habitat without birth term. First we discretize space in a regular equally spaced manner \cite{Bao2020continuum}, \cite{Okubo1980diffusion}, \cite{Turchin1989population}. Let $h$ be the distance between two successive points of the mesh and let $u(x,t)$ be the population density that any individual of the population is at the point $x$ and time $t$. By scaling, we can assume that $0\leq u(x,t)\leq 1$. During a time period $\tau$ an individual which is at the position $x$ and time $t$ can either (see Figure \ref{Fig:Data1}):
\begin{description}
  \item[1.] move to the right of $x$ to the point $x+h$, with probability $R(x,t),$ or
  \item[2.] move to the left of $x$ to the point $x-h$, with probability $L(x,t)$ or
  \item[3.] stay at the position $x$, with probability $N(x,t).$
\end{description}
\begin{figure}[!htb]
   \begin{minipage}{0.80\textwidth}
     \centering
     \includegraphics[width=0.80\linewidth]{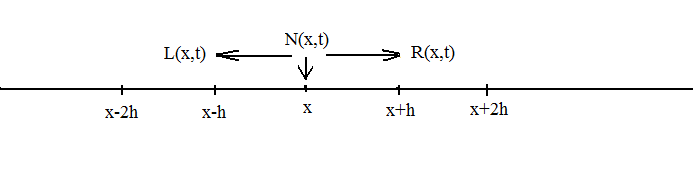}
     \caption{Movement of population}\label{Fig:Data1}
   \end{minipage}\hfill
\end{figure}
Assume that there are no other possibilities of movement we have
\begin{equation*}
 N(x,t) + R(x,t) + L(x,t) = 1.
\end{equation*}
In the simplest case \cite{Okubo1980diffusion}, we assume $N(x,t) = p$ and $R(x,t) = L(x,t) = \frac{1-p}{2}$, the classical random walk or diffusion model is derived as following:
\begin{equation}\label{RDW}
u(x,t+\tau) = N(x,t)u(x,t) + R(x-h,t)u(x-h,t) + L(x+h,t)u(x+h,t)
\end{equation}
Under the constant transferring probability assumption, we have
\begin{equation}\label{lattice-heat}
u(x,t+\tau) = pu(x,t) + \frac{1-p}{2}u(x-h,t) + \frac{1-p}{2}u(x+h,t).
\end{equation}
Expanding all terms in Taylor series, we obtain
\begin{eqnarray*}
&&u(x,t) + \tau \frac{\partial u}{\partial t} + \frac{\tau^2}{2} \frac{\partial^2 u}{\partial t^2}
\\
&&= pu(x,t) + \frac{1-p}{2}[u(x,t) - h\frac{\partial u}{\partial x}(x,t) + \frac{h^2}{2}\frac{\partial^2 u}{\partial x^2}(x,t) - \frac{h^3}{6}\frac{\partial^3 u}{\partial x^3}(x,t) +O(h^4)]
\\
&&\quad + \frac{1-p}{2}[u(x,t) +h\frac{\partial u}{\partial x}(x,t) + \frac{h^2}{2}\frac{\partial^2 u}{\partial x^2}(x,t) + \frac{h^3}{6} \frac{\partial^3 u}{\partial x^3}(x,t) +O(h^4)].
\end{eqnarray*}
With the diffusion approximation assumption ($h^2/\tau \rightarrow C$, when $h,\tau\rightarrow 0$), we derive the classical heat equation in one dimension space
\begin{equation*}
 \frac{\partial u}{\partial t} = \frac{C}{2} \frac{\partial^2 u}{\partial x^2}.
\end{equation*}
However, the transferring probability may depend on the neighboring density information (\cite{KA1}, \cite{Bao2014traveling}, \cite{Turchin1989population}), a more reasonable assumption is as follows:
\begin{eqnarray*}
R(x,t) &=& K(u(x+h,t)),
\\
L(x,t) &=& K(u(x-h,t)),
\end{eqnarray*}
where $ K(u(x,t))$ measures the probability of movement which depends on the population density.


\par Using the notations above again, the density master equation can be written as follows:
\begin{equation}\label{Mast-eq}
u(x,t+\tau) = N(x,t)u(x,t) + R(x-h,t)u(x-h,t) + L(x+h,t)u(x+h,t).
\end{equation}
By using Taylor series, we obtain the following approximation
\begin{eqnarray*}
u(x,t)+\tau\frac{ \partial u}{dt} &=&N(x,t)u(x,t)+[R(x,t)u(x,t)-h\frac{\partial(R u)}{\partial x}+\frac{h^2}{2}\frac{\partial^2(R u)}{\partial x^2}]
\\
&+&[L(x,t)u(x,t) +h\frac{\partial(Lu)}{\partial x}+\frac{h^2}{2}\frac{\partial^2(Lu)}{\partial x^2}],
\end{eqnarray*}
then we get
\begin{equation*}
\tau\frac{\partial u}{\partial t} =[-h\frac{\partial(Ru)}{\partial x}+\frac{h^2}{2}\frac{\partial^2(Ru)}{\partial x^2}]+ [h\frac{\partial(Lu)}{\partial x}+\frac{h^2}{2}\frac{\partial^2(Lu)}{\partial x^2}].
\end{equation*}
By Setting
\begin{equation*}
\beta(x,t) = R(x,t)-L(x,t) = K(u(x+h,t))-K(u(x-h,t))= 2h\frac{\partial}{\partial x}[K(u(x,t)]+O(h^3)
\end{equation*}
and
\begin{equation*}
 \nu(x,t) =R(x,t)+L(x,t)= K(u(x+h,t))+K(u(x-h,t))= 2K(u(x,t))+O(h^2),
\end{equation*}
we have
\begin{equation*}
 \tau\frac{d u}{d t}= -h\frac{\partial[(R-L)u]}{\partial x}+ \frac{h^2}{2}\frac{\partial^2[(R+L)u]}{\partial x^2}.
\end{equation*}
Now substituting  $\beta$ and $\nu$ in the above equation, we can obtain
\begin{equation*}
 \tau\frac{\partial u}{\partial t}= -2h^2\frac{\partial\{\frac{\partial}{\partial x}[K(u(x,t)]u\}}{\partial x}+ \frac{h^2}{2}\frac{\partial^2[2K(u(x,t))u]}{\partial x^2}+ O(h^3),
\end{equation*}
and assume that $h^2/\tau\rightarrow C>0$ (finite) as $\tau, h\rightarrow 0$, we get the following
\begin{equation*}
 \frac{\partial u}{\partial t}= -2C\frac{\partial\{\frac{\partial}{\partial x}[K(u(x,t))]u\}}{\partial x}+ C\frac{\partial^2[K(u(x,t))u]}{\partial x^2},
\end{equation*}
\begin{equation*}
 \frac{\partial u}{\partial t}= C\frac{\partial}{\partial x}\{-2\frac{\partial}{\partial x}[K(u(x,t))]u+ \frac{\partial[K(u(x,t))u]}{\partial x}\},
\end{equation*}
\begin{equation*}
 \frac{\partial u}{\partial t}= C\frac{\partial }{\partial x}(-2K'\frac{\partial u}{\partial x}u+ K'\frac{\partial u}{\partial x}u+ K\frac{\partial u}{\partial x})= C\frac{\partial}{\partial x}[(K-uK')\frac{\partial u}{\partial x}].
\end{equation*}

\par When higher order terms are kept, which will lead to the following
\begin{equation}\label{eq:3001}
 \frac{\partial u}{\partial t} +\tau \frac{\partial^2 u}{\partial t^2}=\frac{\partial}{\partial x}[(K-uK')\frac{\partial u}{\partial x}]+\frac{h^2}{12}[\frac{\partial^4 u}{\partial x^4}K-u\frac{\partial^4}{\partial x^4}(K(u))],
\end{equation}
the coefficient for the high order term $\frac{\partial^4 u}{\partial x^4}$ is $(K-uK')$ which is the same as the diffusion coefficient for $\frac{\partial u}{\partial x}$. Comparing to the standard Cahn-Hilliad equation or Cahn-Hilliad equation with degenerate mobility coefficient, as far as we know Equation \eqref{eq:3001} is new and the properties of the solution will be investigated in the future.

\par By assuming that $K(u(x,t)) = 1/2(u^2-u^3)$ which means the transfer probability is small when density is small or large. Biologically, there is no migration when species detecting zero (small density indicates the species is under risk) or 1 (high density indicates stong competition) neighboring density. We can see that $0\leq K(u(x,t))\leq 1$ which satisfies the assumption of probability and that $K-uK'= u^2(u-1/2)$, which means aggregation when $0\leq u<1/2$. By plugging this probability in the equation \eqref{Mast-eq} and denoting the discrete density $u(x_j,t)=u^t_j,u(x_j-h,t)=u^t_{j-1},u(x_j+h,t)=u^t_{j+1},\dots$, we obtain the discrete iteration model \eqref{lattice-eq}:
\begin{equation}\label{eq:002}
 u_j^{t+\tau}=u_j^t+\frac{u_j^tu_{j-1}^t}{2}(u_j^t+u_{j-1}^t-1)(u_{j-1}^t-u_j^t)+\frac{u_j^tu_{j+1}^t}{2}(u_j^t+u_{j+1}^t-1)(u_{j+1}^t-u_j^t),
\end{equation}
which is a special finite difference scheme of the following backward-forward parabolic equation:
\begin{equation}\label{Master-eq}
 u_t= [ D(u)u_x]_x \quad (x,t)\in Q_T,
\end{equation}
where $Q_T:=[0,1]\times[0,T]$, and $D(u)\in C^\infty[0,1]$ with
\begin{equation}\label{Diffusion-cond}
 D(u)< 0 \quad \mbox{in} \quad(0,\alpha), \quad D(u)>0 \quad\mbox{in}\quad (\alpha,1).
\end{equation}

\section{Pattern formations in the lattice model}
In this section, we will investigate the dynamic behaviors of system \eqref{lattice-eq} with Dirichlet boundary condition, which will have strike differences to classical discrete heat equation \eqref{lattice-heat} with Dirichlet boundary. First, we introduce dynamical behaviors of classical discrete heat equation without birth term, then follows the dynamical results of the aggregation diffusion discrete equation \eqref{lattice-eq}.

Consider Equation \eqref{lattice-heat} with Dirichlet boundary condition, we have
\begin{equation*}
u(j,t+\tau) = pu(j,t) + \frac{1-p}{2}u(j-1,t) + \frac{1-p}{2}u(j+1,t),\quad j=1,2,\cdots,N-1,
\end{equation*}
where $u(j,t)= u(x_j,t), u(0,t) =u(N,t)=0$. Then we have
\begin{eqnarray*}
\sum_{j=1}^{N-1}u(j,t+\tau) &=&\sum_{j=1}^{N-1}[pu(j,t) + \frac{1-p}{2}u(j-1,t) + \frac{1-p}{2}u(j+1,t)]
\\
&=& p\sum_{j=1}^{N-1}u(j,t) + \frac{1-p}{2}\sum_{j=1}^{N-2}u(j,t) + \frac{1-p}{2}\sum_{j=2}^{N-1}u(j,t)
\\
&=& \sum_{j=1}^{N-1}u(j,t) -\frac{1-p}{2}u(1,t) - \frac{1-p}{2}u(N-1,t)
\\
&\leq& \sum_{j=1}^{N-1}u(j,t),
\end{eqnarray*}
which means the total density is a decreasing function in time, and we have limit
\begin{eqnarray*}
\lim_{t\to\infty}\sum_{j=1}^{N-1}u(j,t) = C &=& \lim_{t\to\infty}\sum_{j=1}^{N-1}u(j,t) - \frac{1-p}{2}\lim_{t\to\infty}[u(1,t) +u(N-1,t)]
\\
C&=& C - \frac{1-p}{2}\lim_{t\to\infty}[u(1,t) +u(N-1,t)]
\end{eqnarray*}
Because the positivity of $u(j,t) (0\leq j\leq N)$, the only possibility is
\begin{equation*}
 \lim_{t\to\infty}u(1,t) = \lim_{t\to\infty}u(N-1,t)=0.
\end{equation*}
Then from Equation \eqref{lattice-heat}, we have
\begin{eqnarray*}
\lim_{t\to\infty}u(1,t+\tau) &=& p\lim_{t\to\infty}u(1,t) + \frac{1-p}{2}\lim_{t\to\infty}u(0,t) + \frac{1-p}{2}\lim_{t\to\infty}u(2,t)
\\
0&=& 0 + 0 + \frac{1-p}{2}\lim_{t\to\infty}u(2,t)
\end{eqnarray*}
and
\begin{equation*}
 \lim_{t\to\infty}u(2,t)=0.
\end{equation*}
By the iteration method, we obtain
\begin{equation*}
 \lim_{t\to\infty}u(i,t)= \lim_{t\to\infty}u(x_i,t) =0,\quad i=0,1,\cdots,N,
\end{equation*}
which can be views as the discrete version of Theorem \ref{Main-thm2} in section 4. Numerical simulation see Figure \ref{Fig:Data2} where the transferring probability $p=1/2$, the time period $\tau=0.1$, and $h=1/200$.
\begin{figure}[!htb]
   \begin{minipage}{1.00\textwidth}
     \centering
     \includegraphics[width=1.00\linewidth]{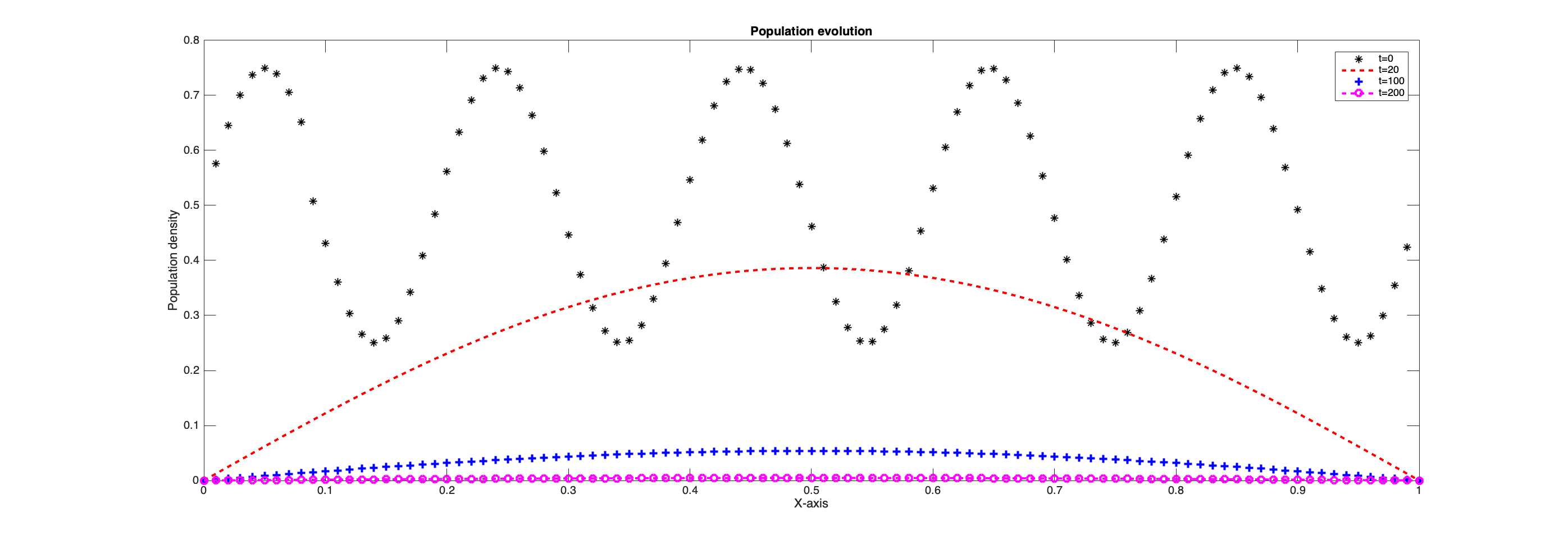}
     \caption{Initial density $u(x,0)=\frac{1}{4}\sin(10\pi)+\frac{1}{2}$}\label{Fig:Data2}
   \end{minipage}\hfill
\end{figure}

In the following, we will mainly focus on the boundedness, conservation, asymptotic behaviors of Equation \eqref{lattice-eq} with Dirichlet boundary condition.

In the derivation of equation \eqref{lattice-eq} and \eqref{continu-eq}, we assume that population density has the range in the interval $[0,1]$ and it can be proved rigorously in the following statement with the initial density in $[0,1]$.
\begin{thm}[Bounded-ness]\label{Main-thm5}
 Suppose the initial solution $0\leq u(j,0)\leq 1, j=0,\dots,N$, with Dirichlet boundary condition, the solution of the equation \eqref{lattice-eq} is bounded in $[0,1]$.
\end{thm}
\begin{proof}
Rewrite Equation \eqref{lattice-eq} as
\begin{equation}\label{lattice-eq-2}
 u_j^{t+\tau}=\frac{u_j^t}{2}+\frac{u_j^tu_{j-1}^t}{2}(u_j^t+u_{j-1}^t-1)(u_{j-1}^t-u_j^t)+ \frac{u_j^t}{2}+\frac{u_j^tu_{j+1}^t}{2}(u_j^t+u_{j+1}^t-1)(u_{j+1}^t-u_j^t),
\end{equation}
and then consider the auxiliary function
\begin{equation}\label{Aux-eq}
f(x,y) = x + xy(x+y-1)(y-x), \quad \forall x,y\in [0,1].
\end{equation}
Then by proving the maximum and minimum of $f(x,y)$ is in the domain $[0,1]$ and we can prove the statement.
The detailed arguments are similar to the Neumann boundary or non-flux boundary condition case in \cite[Theorem 3.1]{Bao2020continuum}, we omit the details here.

\begin{thm}[Conservation]\label{Main-thm6}
Given the initial solution $0\leq u(j,0)\leq 1, j=0,\dots,N$, the total density of system \eqref{lattice-eq} is conservative with Dirichlet boundary condition:
\begin{equation}\label{Conserv}
 \sum_{j=1}^{N-1} u(j,t+\tau) = \sum_{j=1}^{N-1} u(j,t)
\end{equation}
\end{thm}
\end{proof}
\begin{proof} For the Dirichlet boundary condition, we have
\begin{equation*}
 u(0,t)=0,\quad u(N,t)=0.
\end{equation*}
 From Equation \eqref{lattice-eq}, we have
\begin{equation}\label{lattice-eq-3}
 u_j^{t+\tau}=u_j^t+C_j^t(u_{j-1}^t-u_j^t)+C_{j+1}^t(u_{j+1}^t-u_j^t),
\end{equation}
where
\begin{equation}\label{C-coef}
 C_j^t = \frac{u_j^tu_{j-1}^t}{2}(u_j^t+u_{j-1}^t-1).
\end{equation}
Because $C_1^t=C_N^t=0$, then we get
\begin{eqnarray*}
 \sum_{j=1}^{N-1} u_j^{t+\tau}&=&\sum_{j=1}^{N-1} u_j^{t} + C_1^t(u_0^t-u_1^t)+C_2^t(u_2^t-u_1^t)
 \\
 &&\quad +C_2^t(u_1^t-u_2^t)+C_3^t(u_3^t-u_2^t)+\dots+
 \\
 &&\quad +C_{N-1}^t(u_{N-2}^t-u_{N-1}^t)+C_N^t(u_N^t-u_{N-1}^t),
 \\
 &=& \sum_{j=1}^{N-1} u_j^{t} + C_1^t(u_0^t-u_1^t) + C_N^t(u_N^t-u_{N-1}^t) = \sum_{j=1}^{N-1}u_j^{t},
\end{eqnarray*}
and the density in the lattice model \eqref{lattice-eq} is conservative.
\end{proof}

In the following, we have the maximal principal and asymptotic results of Equation \eqref{lattice-eq} with initial solution in the forward or diffusion domain ($1/2=\alpha\leq u(j,0)\leq 1$) which is similar to Theorem \ref{Main-thm1}, \ref{thm2} for Equation \eqref{continu-eq} with Neumann boundary condition in section 4. For the general initial solution case, we will see in the special 5
 lattice points case that a rich dynamical behaviors exist.

\begin{thm}[Maximum principle]\label{thm4}
Suppose $1/2=\alpha\leq u(j,0)\leq 1,$ $ j=1,\dots,N-1$, then the solution of Equation \eqref{lattice-eq} with Dirichlet boundary satisfies
\begin{equation}
\min_{1\leq j\leq N-1}u(j,0)\leq u(j,t)\leq \max_{1\leq j\leq N-1}u(j,0) \quad \forall t\geq 0.
\end{equation}
\end{thm}
\begin{proof} Suppose $u(k_1,t)=\max_{1\leq j\leq N-1}u(j,t)$, $u(k_2,t)=\min_{1\leq j\leq N-1}u(j,t)$. In the special case the maximal density is close the boundary point, for example $k_1=1$, then from Equation \eqref{lattice-eq} we have
\begin{eqnarray*}\label{max-value}
u_1^{t+\tau} &=& u_1^t + \frac{u_{1}^tu_{2}^t}{2}(u_{1}^t+u_{2}^t-1)(u_{2}^t-u_{1}^t)
\\
&\leq & u_1^t.
\end{eqnarray*}
Similar result exists when the minimum density is close the boundary points, which means
\begin{equation*}
 u(k_2,t+\tau) \geq u(k_2,t).
\end{equation*}
In the general case, the maximal density point(s) is in the interior of the lattice points, then we have
\begin{eqnarray}\label{eq:015}
 u(k_1,t+\tau)-u(k_1,t)&=&\frac{u_{k_1}^tu_{k_1-1}^t}{2}(u_{k_1}^t+u_{k_1-1}^t-1)(u_{k_1-1}^t-u_{k_1}^t)\nonumber
 \\
 &+&\frac{u_{k_1}^tu_{k_1+1}^t}{2}(u_{k_1}^t+u_{k_1+1}^t-1)(u_{k_1+1}^t-u_{k_1}^t),
\end{eqnarray}
with $u_{k_1}^t+u_{k_1-1}^t-1\geq 0,u_{k_1}^t+u_{k_1+1}^t-1\geq 0$ and $u_{k_1-1}^t-u_{k_1}^t\leq0, u_{k_1+1}^t-u_{k_1}^t\leq 0$, which leads to
\begin{equation*}
 u(k_1,t+\tau)-u(k_1,t)\leq 0.
\end{equation*}
The same idea can also be used for the minimum value point and get
\begin{equation*}
 u(k_2,t+\tau)-u(k_2,t)\geq 0.
\end{equation*}
For the ordinary point $u(i,t)$
\begin{eqnarray*}
 u(i,t+\tau)-u(i,t)&=&\frac{u_i^tu_{i-1}^t}{2}(u_i^t+u_{i-1}^t-1)(u_{i-1}^t-u_i^t)+\frac{u_i^tu_{i+1}^t}{2}(u_i^t+u_{i+1}^t-1)(u_{i+1}^t-u_i^t)
\\
&\leq& 1/2[(u_{k_1}^t-u_i^t)]+1/2[(u_{k_1}^t-u_i^t)]
\\
&\leq& u_{k_1}^t-u_i^t.
\end{eqnarray*}
Again
\begin{eqnarray*}
 u(i,t+\tau)-u(i,t)&=&\frac{u_i^tu_{i-1}^t}{2}(u_i^t+u_{i-1}^t-1)(u_{i-1}^t-u_i^t)+\frac{u_i^tu_{i+1}^t}{2}(u_i^t+u_{i+1}^t-1)(u_{i+1}^t-u_i^t)
\\
&\geq& 1/2[(u_{k_2}^t-u_i^t)]+1/2[(u_{k_2}^t-u_i^t)]
\\
&\geq& u_{k_2}^t-u_i^t.
\end{eqnarray*}
Then by the iteration method, we obtain the following maximum principal of Equation \eqref{lattice-eq} with Dirichlet boundary
\begin{equation*}
 \min_{1\leq j\leq N}u(j,0)\leq u(j,t)\leq\max_{1\leq j\leq N}u(j,0).
\end{equation*}
\end{proof}

In the following, we investigate the monotonicity of solution of Equation \eqref{lattice-eq} with Dirichlet boundary condition. Here we only focus on the increasing initial solution case and the ideas for the decreasing initial solution case are the same.

\begin{thm}[Monotonicity]\label{Monotone}
Suppose the initial solution of Equation \eqref{lattice-eq} satisfies
\begin{equation}\label{Initial-mono}
1/2=\alpha\leq u(1,0)\leq u(3,0)\leq \cdots\leq u(N-1,0)\leq 1,
\end{equation}
then we have
\begin{equation}\label{general-mono}
1/2=\alpha\leq u(1,t)\leq u(2,t)\leq \cdots\leq u(N-1,t)\leq1,\quad \forall t>0.
\end{equation}
\end{thm}
\begin{proof}
The initial solution is increasing in the interior of the domain, we have
\begin{equation}\label{Initial}
 1/2=\alpha<  u(1,0)\leq\dots\leq u(N-1,0)< 1.
\end{equation}
From Equation \eqref{lattice-eq}, we have
\begin{equation}\label{Var-eq}
 \Delta U^{t+\tau}=[C]^t\Delta U^t,
\end{equation}
where
\begin{equation*}
 \Delta U^{t+\tau}=[(u_2^{t+\tau}-u_1^{t+\tau}),\dots,(u_{N-1}^{t+\tau}-u_{N-2}^{t+\tau})]^T,
\end{equation*}
\begin{equation*}
[C]^t=
\left(
         \begin{array}{ccccccccc}
           1-2C_2^t & C_3^t & 0 &  \dots & 0 & 0 & 0 \\
            C_2^t & 1-2C_3^t & C_4^t & \dots & 0 & 0 & 0  \\
           \vdots & \vdots & \vdots & \ddots & \vdots & \vdots & \vdots \\

           0 & 0 & 0 & \dots &C_{N-3}^t & 1-2C_{N-2}^t & C_{N-1}^t  \\
           0 & 0 & 0  & \dots & 0 &   C_{N-2}^t &1-2 C_{N-1}^t\\

         \end{array}
       \right)
\end{equation*}
and
\begin{equation}\label{C-condition}
 C_j^t=\frac{u_j^tu_{j-1}^t}{2}(u_j^t+u_{j-1}^t-1),\quad j=2,3,\cdots,N-1.
\end{equation}
From Theorem \ref{thm4}, we have $1/2\leq u_j^t\leq 1$, $0\leq C_j^t=\frac{u_j^tu_{j-1}^t}{2}(u_j^t+u_{j-1}^t-1)\leq 1/2$, so $[C]^t$ is a nonnegative matrix. Initially $\Delta U^0\geq 0$, and by the iteration method we have $\Delta U^t\geq 0$ for all $t\geq 0$ which leads to the conservation of monotonicity of $u_j^t$ in $1\leq j\leq N-1$ for all $t>0$.
\end{proof}

\begin{rk}\label{mono-eq}
Initially, we can suppose $1/2\leq u(1,0)\leq u(2,0)\leq \cdots\leq u(N-1,0)\leq 1$. Except for two trivial cases $1/2= u(1,0)= u(2,0)= \cdots= u(N-1,0)$ and $ u(1,0)= u(2,0)= \cdots= u(N-1,0)=1$, by the iteration equation \eqref{lattice-eq}, the discrete solution will have the relation
\begin{equation}\label{Initial-mono}
1/2< u(1,t^*)\leq u(2,t^*)\leq \cdots\leq u(N-1,t^*)< 1,
\end{equation}
where $t^* >0$ and the initial relation becomes to \eqref{Initial-mono}.
\end{rk}

\begin{thm}[Asymptotic behavior]\label{Asymp-thm}
Suppose $1/2=\alpha\leq u(j,0)\leq 1,$ $ j=1,\dots,N-1$, $u(0,0)= u(N,0)=0$, then Equation \eqref{lattice-eq} has the following asymptotic convergence result:
\begin{equation}\label{Asymp-beha}
  \lim_{t\rightarrow \infty} u(j,t)=\frac{1}{N-1}\sum_{j=1}^{N-1} u(j,0).
 \end{equation}
\end{thm}
\begin{proof}
In the trivial cases $ u(1,0) = u(2,0)=\dots=u(N-1,0) =1/2$ and $ u(1,0) = u(2,0)=\dots=u(N-1,0) =1$, it is easy to obtain the relation \eqref{Asymp-beha}. In the general case, by maximum principal theorem \ref{thm4} we have
\begin{equation}
1/2\leq u(j,t)\leq 1,\quad j=1,\dots,N-1.
\end{equation}

From Equation \eqref{Var-eq}, we have
\begin{eqnarray*}
 u_2^{t+\tau} - u_1^{t+\tau} &=& (1 -2C_2^t)(u_2^t-u_1^t) + C_3^t(u_3^t-u_2^t),
 \\
 u_3^{t+\tau} - u_2^{t+\tau} &=& C_2^t(u_2^t-u_1^t) +(1- 2C_3^t)(u_3^t-u_2^t) + C_4^t(u_4^t-u_3^t),
 \\
\vdots&&\vdots
\\
u_i^{t+\tau} - u_{i-1}^{t+\tau} &=& C_{i-1}^t(u_{i-1}^t-u_{i-2}^t) +(1- 2C_i^t)(u_i^t-u_{i-1}^t) + C_{i+1}^t(u_{i+1}^t-u_i^t),
\\
\vdots&&\vdots
\\
u_{N-1}^{t+\tau} - u_{N-2}^{t+\tau} &=& C_{N-2}^t(u_{N-2}^t-u_{N-3}^t) + (1 -2C_{N-1}^t)(u_{N-1}^t-u_{N-2}^t).
\end{eqnarray*}
By the maximum principle Theorem \ref{thm4}, we have $1/2\leq u_j^t\leq 1$, $0\leq C_j^t=\frac{u_j^tu_{j-1}^t}{2}(u_j^t+u_{j-1}^t-1)\leq 1/2$, Then
\begin{eqnarray}
\sum_{j=2}^{N-1}|u_{i}^{t+\tau} - u_{i-1}^{t+\tau}| &\leq& \sum_{j=2}^{N-1}|u_{i}^{t} - u_{i-1}^{t}| - C_2^t|u_2^t-u_1^t|- C_{N-1}^t|u_{N-1}^t-u_{N-2}^t|\nonumber
\\
&\leq& \sum_{j=2}^{N-1}|u_{i}^{t} - u_{i-1}^{t}|,\quad \forall t\geq 0,\label{Var-dec}
\end{eqnarray}
and $\sum_{j=2}^{N-1}|u_{i}^{t} - u_{i-1}^{t}|$ is a decreasing function of time $t$ which leads to the convergence result:
\begin{equation}\label{Var-asymp}
\lim_{t\to\infty}\sum_{j=2}^{N-1}|u_{i}^{t+\tau} - u_{i-1}^{t+\tau}| = \lim_{t\to\infty}\sum_{j=2}^{N-1}|u_{i}^{t} - u_{i-1}^{t}| - \lim_{t\to\infty}C_2^t|u_2^t-u_1^t|- \lim_{t\to\infty}C_{N-1}^t|u_{N-1}^t-u_{N-2}^t|.
\end{equation}
Because $0\leq C_i^t\leq1/2$, $(2\leq i\leq N-1)$, we can obtain that
\begin{equation}\label{Var-lim}
\lim_{t\to\infty}C_2^t|u_2^t-u_1^t|= \lim_{t\to\infty}C_{N-1}|u_{N-1}^t-u_{N-2}^t|=0.
\end{equation}
In the case
\begin{equation*}
\lim_{t\to\infty}C_2^t=\lim_{t\to\infty}\frac{u_2^tu_1^t}{2}(u_2^t +u_1^t -1)=0,
\end{equation*}
by the maximum principle theorem \ref{thm4}, the only possibility is
\begin{equation}
\lim_{t\to\infty}u_1^t = \lim_{t\to\infty}u_2^t = 1/2.
\end{equation}
Then by the iteration equation \eqref{lattice-eq} and the maximum principle theorem \ref{thm4}, we have
\begin{equation}
\lim_{t\to\infty}u_1^t = \lim_{t\to\infty}u_2^t =\dots=\lim_{t\to\infty}u_{N-1}^t=1/2
\end{equation}
which is the trivial case $u(1,0)=u(2,0)=\dots=u(N-1,0)=1/2$ by the conservation theorem \ref{Main-thm6}. The same result holds for $\lim_{t\to\infty}C_{N-1}^t =0$.

In the case that
\begin{equation}
\lim_{t\to\infty}|u_2^t-u_1^t|=0,
\end{equation}
and  we assume $\lim_{t\to\infty} u_1^t =\lim_{t\to\infty} u_2^t >1/2$, otherwise it will go back to the trivial case $u(1,0)=u(2,0)=\dots=u(N-1,0)=1/2$.
Then by iteration Equation \eqref{lattice-eq}, the maximum principle theorem \ref{thm4}, and the conservation Theorem \ref{Main-thm6} we have
\begin{equation}\label{Var-lim}
\lim_{t\to\infty}|u_3^t-u_2^t|= \lim_{t\to\infty}|u_4^t-u_3^t|=\cdots=\lim_{t\to\infty}|u_{N-1}^t-u_{N-2}^t|=0,
\end{equation}
which leads to Equation \eqref{Asymp-beha}. Numerical simulation see Figure \ref{Fig:Data3} where the time period $\tau =0.1$ and $h=1/500$.
\begin{figure}[!htb]
   \begin{minipage}{1.00\textwidth}
     \centering
     \includegraphics[width=1.00\linewidth]{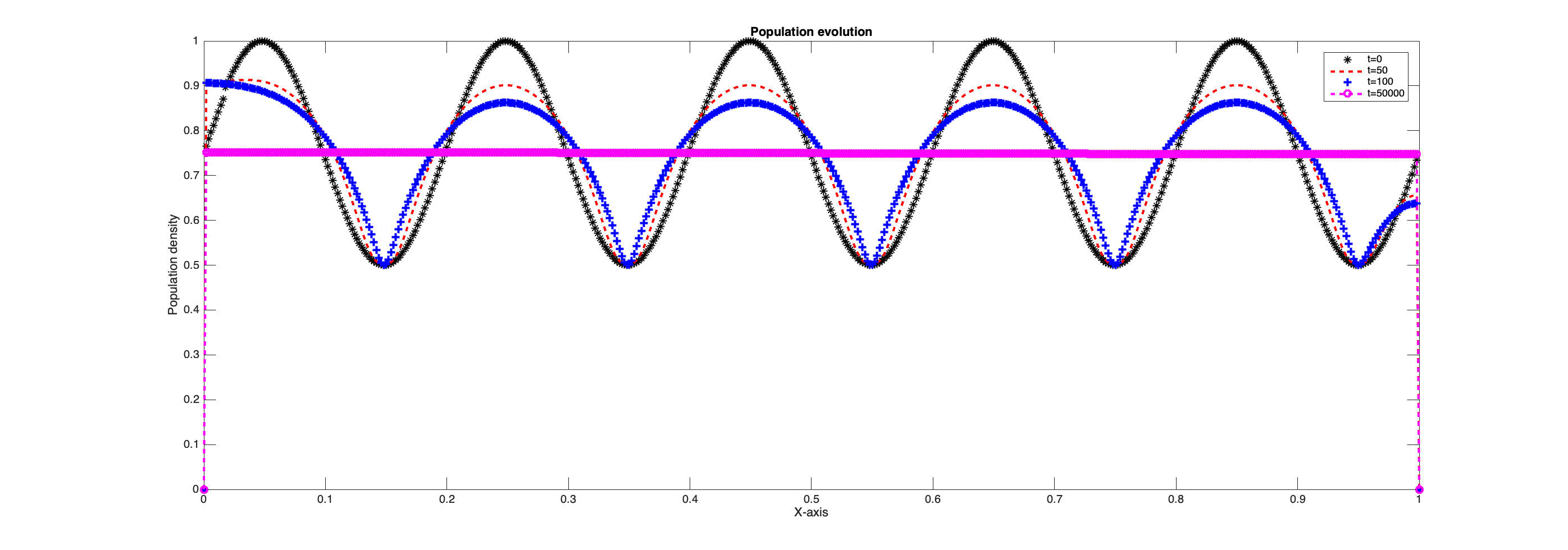}
     \caption{Initial density $u(x,0)=\frac{1}{4}\sin(10\pi)+\frac{3}{4}$ }\label{Fig:Data3}
   \end{minipage}\hfill
\end{figure}

\end{proof}

\subsection{Asymptotic behaviors under special case when $N\leq5$}

In this section, we consider the asymptotic behaviors of the solution of \eqref{lattice-eq} in a special case when $N=5$ with time period $\tau =0.1$ and $u(0,t)=u(4,t)=0$ for $t\geq 0$ which corresponding to the Dirichlet boundary condition. In the case $N=3,4$ with Dirichlet boundary condition, the asymptotic behaviors of the solutions are easy to be obtained.
We rewrite the lattice model \eqref{lattice-eq} in the following way:
\begin{eqnarray}
 u_1^{t+\tau}&=&u_1^t+\frac{u_2^tu_{1}^t}{2}(u_1^t+u_{2}^t-1)(u_{2}^t-u_1^t),\label{Lattice-1}
 \\
 u_2^{t+\tau}&=&u_2^t+C_2^t(u_{1}^t-u_2^t)+C_3^t(u_{3}^t-u_2^t),\label{Lattice-2}
 \\
 u_1^{t+\tau}+u_2^{t+\tau}&=&u_1^t+u_2^t+\frac{u_2^tu_{3}^t}{2}(u_2^t+u_{3}^t-1)(u_{3}^t-u_2^t),\label{Lattice-3}
 \\
 u_2^{t+\tau}-u_1^{t+\tau}&=&(u_2^t-u_1^t)(1-2C_2^t)+\frac{u_2^tu_{3}^t}{2}(u_2^t+u_{3}^t-1)(u_{3}^t-u_2^t),\label{Lattice-4}
 \\
 u_3^{t+\tau}&=&u_3^t+\frac{u_3^tu_{2}^t}{2}(u_3^t+u_{2}^t-1)(u_{2}^t-u_3^t),\label{Lattice-5}
 \\
 u_2^{t+\tau}+u_3^{t+\tau}&=&u_2^t+u_3^t+\frac{u_2^tu_{1}^t}{2}(u_2^t+u_{1}^t-1)(u_{1}^t-u_2^t),\label{Lattice-6}
 \\
 u_3^{t+\tau}-u_2^{t+\tau}&=&(u_3^t-u_2^t)(1-2C_3^t)+\frac{u_2^tu_{1}^t}{2}(u_2^t+u_{1}^t-1)(u_{2}^t-u_1^t).\label{Lattice-7}
\end{eqnarray}
In the following, we consider the asymptotic behaviors of Equation \eqref{lattice-eq} with different initial solutions. By the bounded-ness Theorem \ref{Main-thm5}, the solution of \eqref{lattice-eq} is bounded in the domain $[0,1]$ given $0\leq u(i,0) \leq 1, i=0,1,\cdots, 4$.


\par {\bf Case 1}: $u(1,0)+u(2,0)<1, u(2,0)+u(3,0)<1$.
\par Subcase 1: $u(2,0)$ is the minimum point.
From equations \eqref{Lattice-1}-\eqref{Lattice-6}, we have
 $u_1^t,u_3^{t}$ are increasing, $u_2^t,u_1^{t}+u_2^{t},u_2^{t}+u_3^{t}$ are decreasing and $u(j,t)$ are bounded in $[0,1]$ which leads to the existence of the limit of each term and the following asymptotic behaviors of the solution (see Figure \ref{Fig:Data4})
\begin{eqnarray*}
 \lim_{t\rightarrow \infty}u_1^{t}&=&\overline{u}_1,
 \\
 \lim_{t\rightarrow \infty}u_2^{t}&=&\overline{u}_2=0,
 \\
 \lim_{t\rightarrow \infty}u_3^{t}&=&\overline{u}_3.
\end{eqnarray*}
\begin{figure}[!htb]
   \begin{minipage}{1.00\textwidth}
     \centering
     \includegraphics[width=1.00\linewidth]{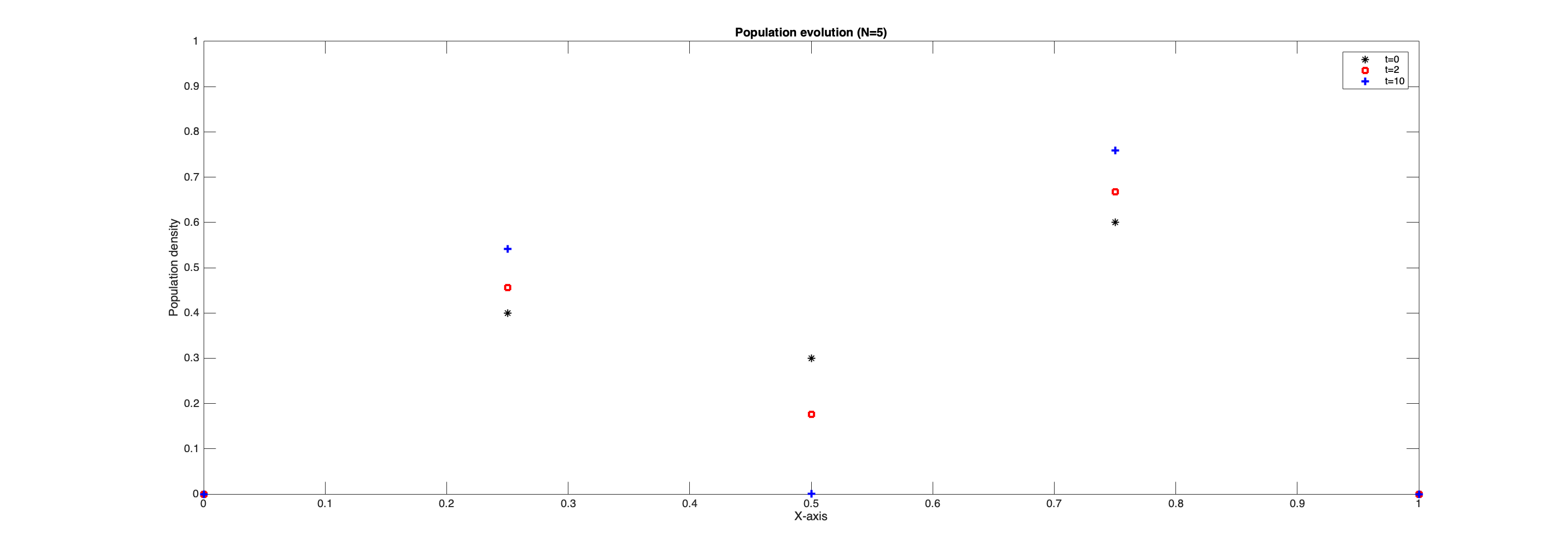}
     \caption{Initial density $u(1,0)=.4, u(2,0)=.3,u(3,0)=.6$ }\label{Fig:Data4}
   \end{minipage}\hfill
\end{figure}
\par Subcase 2: $u(1,0)+u(2,0)+u(3,0)<1$ and $u(2,0)$ is the maximum point.
 \par
From equations \eqref{Lattice-1}-\eqref{Lattice-6}, we have
 $u_1^t,u_3^{t}$ are decreasing, $u_2^t$ is increasing and $u(j,t)$ are bounded in $[0,1]$ which leads to the existence of the limit of each term and the following asymptotic behaviors of the solution (see Figure \ref{Fig:Data5})
\begin{eqnarray*}
 \lim_{t\rightarrow \infty}u_1^{t}&=&\overline{u}_1=0,
 \\
 \lim_{t\rightarrow \infty}u_2^{t}&=&\overline{u}_2=u(1,0)+u(2,0)+u(3,0),
 \\
 \lim_{t\rightarrow \infty}u_3^{t}&=&\overline{u}_3=0.
\end{eqnarray*}
\begin{figure}[!htb]
   \begin {minipage}{1.00\textwidth}
     \centering
     \includegraphics[width=1.00\linewidth]{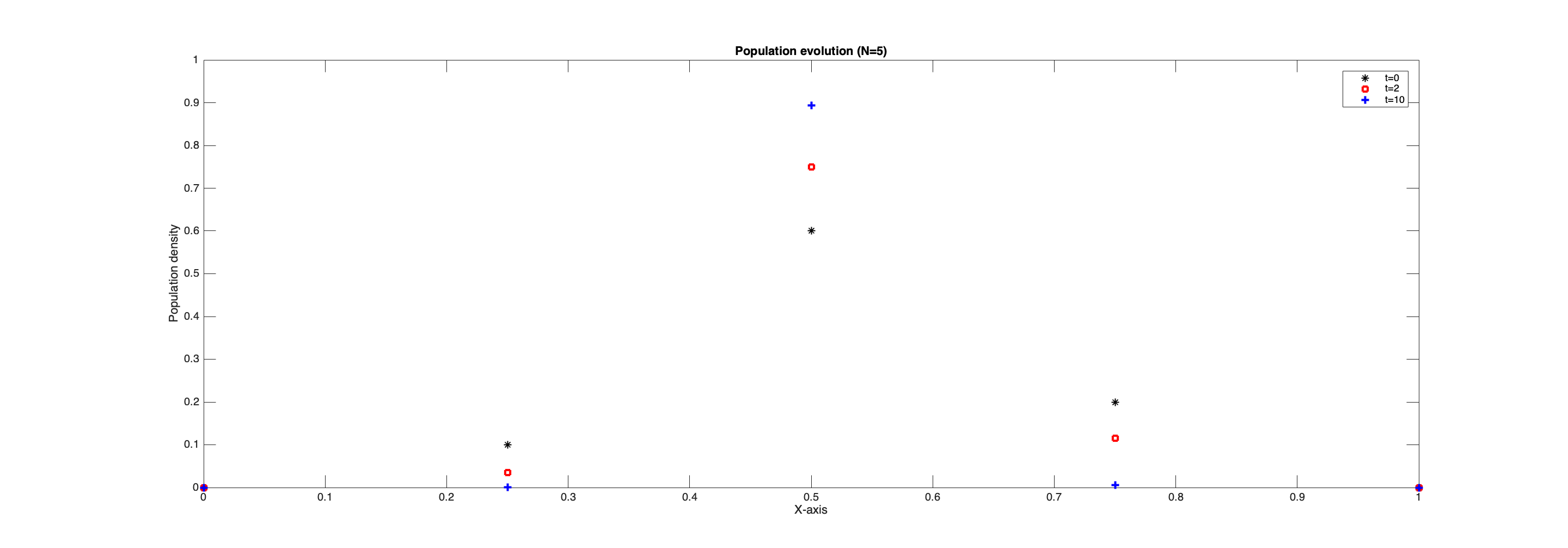}
     \caption{Initial density $u(1,0)=.1, u(2,0)=.6,u(3,0)=.2$}\label{Fig:Data5}
   \end{minipage}
\end{figure}
\par Subcase 3: $u(1,0)+u(2,0)+u(3,0)<1$ and $u(3,0)$ is the maximum point, $u(1,0)$ is the minimum point, otherwise it is the subcase 1 (When $u(1,0)$ is the maximum, the idea is the same as $u(3,0)$ is the maximum.).
\par
From equations \eqref{Lattice-1}-\eqref{Lattice-6}, we assume that $u^{t+\tau}_3>u^{t+\tau}_2>u^{t+\tau}_1$, otherwise it becomes to the subcase 1 or the subcase 2,  and the limits exist. In this case $u_3^t$ is increasing and $u(j,t)$ are bounded in $[0,1]$ which leads to the existence of the limit of each term and the following asymptotic behaviors of the solution (see Figure \ref{Fig:Data6})
\begin{eqnarray*}
 \lim_{t\rightarrow \infty}u_1^{t}&=&\overline{u}_1 =0,
 \\
 \lim_{t\rightarrow \infty}u_2^{t}&=&\overline{u}_2=0,
 \\
 \lim_{t\rightarrow \infty}u_3^{t}&=&\overline{u}_3=u(1,0)+u(2,0)+u(3,0).
\end{eqnarray*}
\begin{figure}[!htb]
\begin{minipage}{1.00\textwidth}
     \centering
     \includegraphics[width=1.00\linewidth]{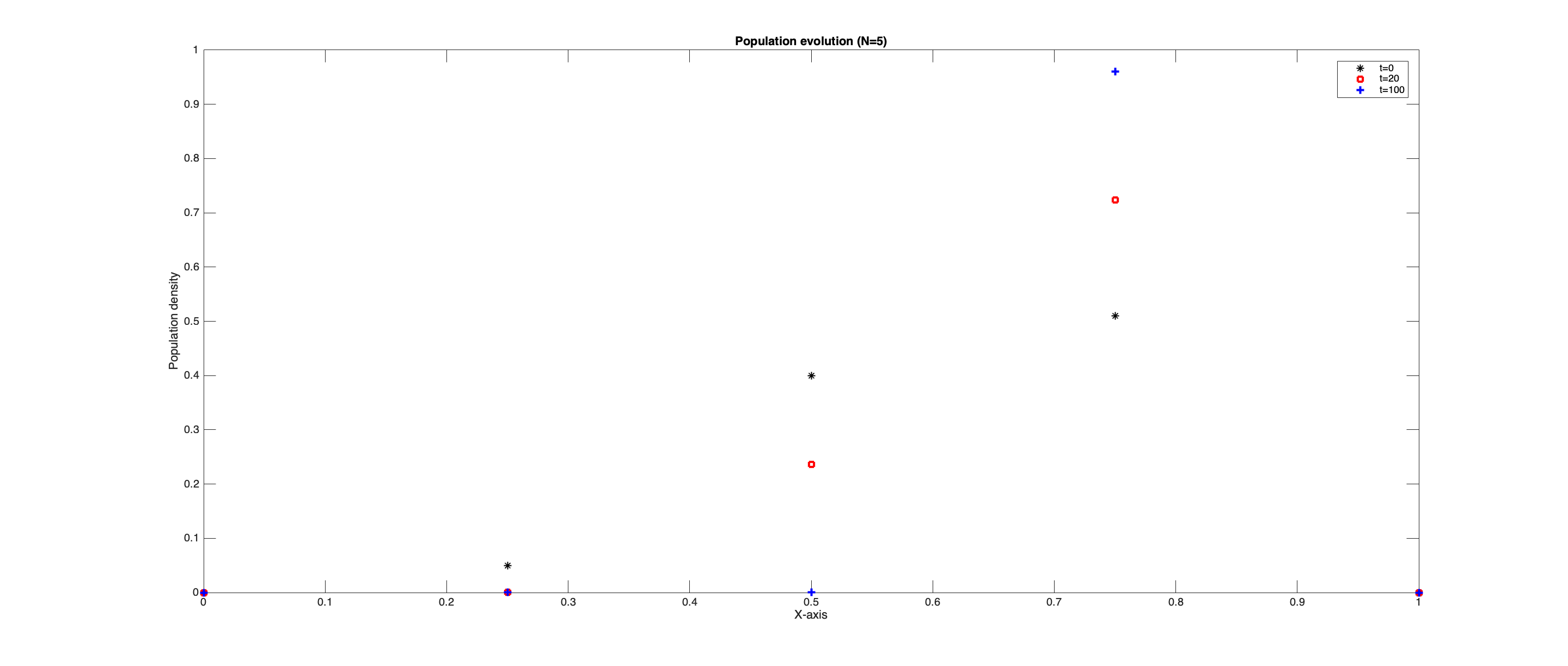}
     \caption{Initial density $u(1,0)=.05, u(2,0)=.4,u(3,0)=.51$ }\label{Fig:Data6}
   \end{minipage}\hfill
\end{figure}
\par Subcase 4: $1\leq u(1,0)+u(2,0)+u(3,0)$ and $u(3,0)$ is the maximum (When $u(1,0)$ is the maximum, the idea is the same as $u(3,0)$ is the maximum.).
\par If $u(2,0)$ is the minimum, it goes to the subcase 1. So we consider the case when $u(3,0)>u(2,0)>u(1,0)$. From equations \eqref{Lattice-1}-\eqref{Lattice-6}, we have
 $u_1^t,u_1^{t}+u_2^{t}$ are decreasing, $u_3^t,u_2^{t}+u_3^{t}$ are increasing first. If $u_1^{t}+u_2^{t} \leq 1$ and $u_2^{t}+u_3^{t}\leq 1$ for all time $t$, and $u(2,t)$ is the minimum after some time point $t^*$, it goes back to subcase 1 and we have the asymptotic convergent result. In other case, $u_1^{t}+u_2^{t} \leq 1$ and $u_2^{t}+u_3^{t}\leq 1$, and $u(3,t) > u(2,t) > u(1,t)$ for all time $t$, then by the monotonicity of $u(1,t), u(3,t)$ and bounded-ness properties of them, we also have the convergent results
 \begin{equation*}
\lim_{t\rightarrow\infty}u_1^t=0,\quad \lim_{t\rightarrow\infty}u_2^t=\lim_{t\rightarrow\infty}u_3^t = \frac{1}{2}[u(1,0) + u(2,0) +u(3,0)].
\end{equation*}
  We have $1\leq u(1,0)+u(2,0)+u(3,0)$, when $ u_2^{t^*}+u_3^{t^*} > 1$  and $u_1^{t^*}+u_2^{t^*} < 1$ at some point $t^*$, it becomes to case 3 below and we have the asymptotic limits.
\par {\bf Case 2}: $u(1,0)+u(2,0)>1$ and $u(2,0)+u(3,0)>1$.
\par Subcase 1: $u(2,0)$ is the minimum point. From Equation \eqref{Lattice-3} and \eqref{Lattice-6}, we can get $u_1^t+u_2^t>1, u_2^t+u_3^t>1$ for all $t>0$. So we can get the following
\begin{eqnarray*}
 |u_2^{t+\tau}-u_1^{t+\tau}|&\leq& (1-2C_2^t)|u_2^t-u_1^t|+ C_3^t|u_3^t-u_2^t|,
 \\
  |u_3^{t+\tau}-u_2^{t+\tau}|&\leq& (1-2C_3^t)|u_3^t-u_2^t|+C_2^t|u_2^t-u_1^t|,
\end{eqnarray*}
and then
\begin{eqnarray*}
|u_2^{t+\tau}-u_1^{t+\tau}| + |u_3^{t+\tau}-u_2^{t+\tau}| &\leq& (1-C_2^t)|u_2^t-u_1^t| + (1-C_3^t)|u_3^t-u_2^t|
\\
&<& |u_2^t-u_1^t| + |u_3^t-u_2^t|
\end{eqnarray*}
which equivalent to the decreasing of the total variation of population density in time $t$ and leads to the convergence of the asymptotic behaviors of $u(1,t),u(2,t)$ and $u(3,t)$ (see Figure \ref{Fig:Data7}).
\begin{equation*}
\lim_{t\rightarrow\infty}u_1^t=\lim_{t\rightarrow\infty}u_2^t=\lim_{t\rightarrow\infty}u_3^t = \frac{1}{3}[u(1,0) + u(2,0) +u(3,0)].
\end{equation*}
\begin{figure}[!htb]
   \begin{minipage}{1.00\textwidth}
     \centering
     \includegraphics[width=1.00\linewidth]{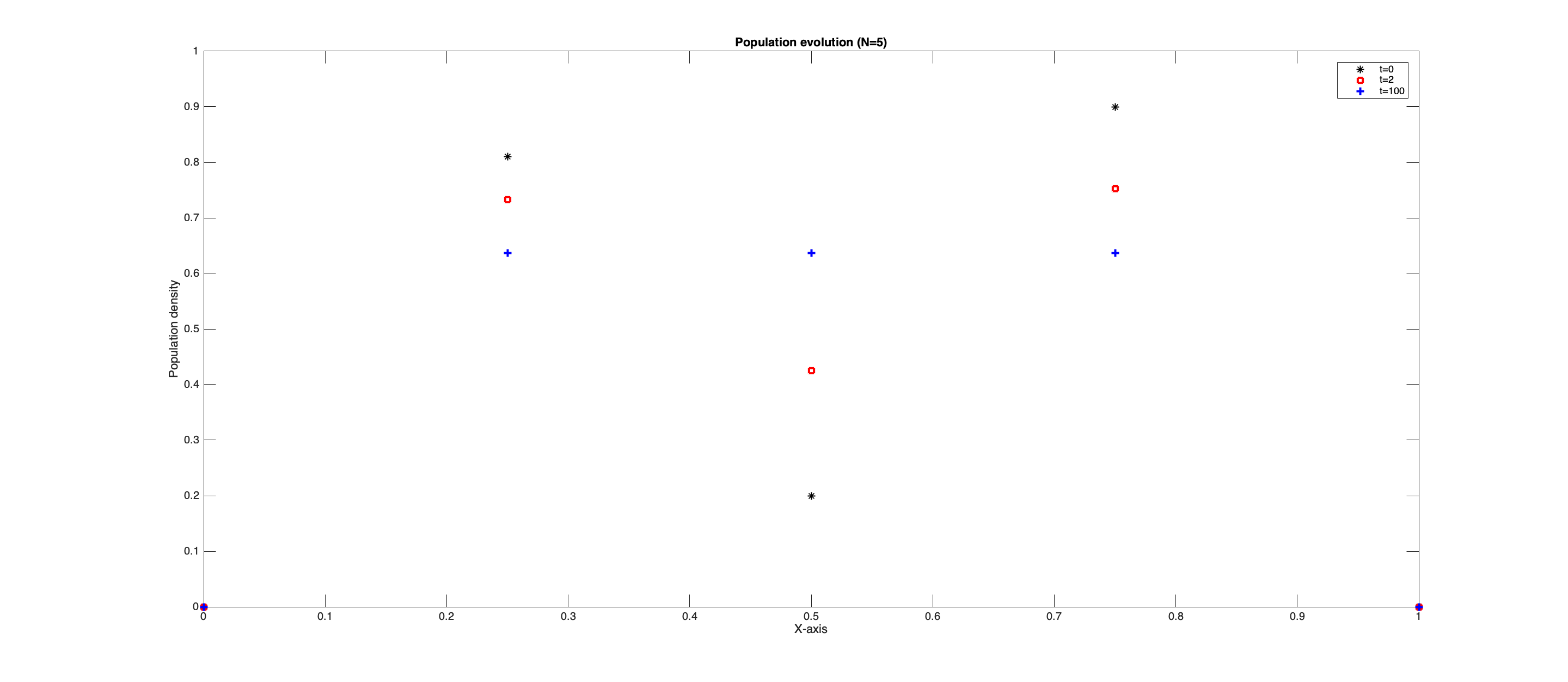}
     \caption{Initial density $u(1,0)=.81, u(2,0)=.2,u(3,0)=.9$ }\label{Fig:Data7}
   \end{minipage}\hfill
\end{figure}

\par Subcase 2: $u(3,0)>u(2,0)>u(1,0)$. From Equation \eqref{Lattice-1}-\eqref{Lattice-7} we have that $u(1,t)+u(2,t)$ is increasing first and $u(2,t)+u(3,t)$ is decreasing but always larger than 1 ($u(1,0) + u(2,0)> 1, u(2,0) + u(3,0) >1$ indicates $u(2,t), u(3,t)$ in domain $[1/2, 1]$). From theorem \ref{thm4}, Equation \eqref{Lattice-4}, \eqref{Lattice-7}, we have $u(1,t)\leq u(2,t)\leq u(3,t)$ for all $t\geq 0$, $u(1,t)$ is increasing and $u(3,t)$ is decreasing in time $t$, and we have (see Figure \ref{Fig:Data8})
\begin{equation*}
\lim_{t\rightarrow\infty}u_1^t=\lim_{t\rightarrow\infty}u_2^t=\lim_{t\rightarrow\infty}u_3^t = \frac{1}{3}[u(1,0) + u(2,0) +u(3,0)].
\end{equation*}
\begin{figure}[!htb]
   \begin{minipage}{1.00\textwidth}
     \centering
     \includegraphics[width=1.00\linewidth]{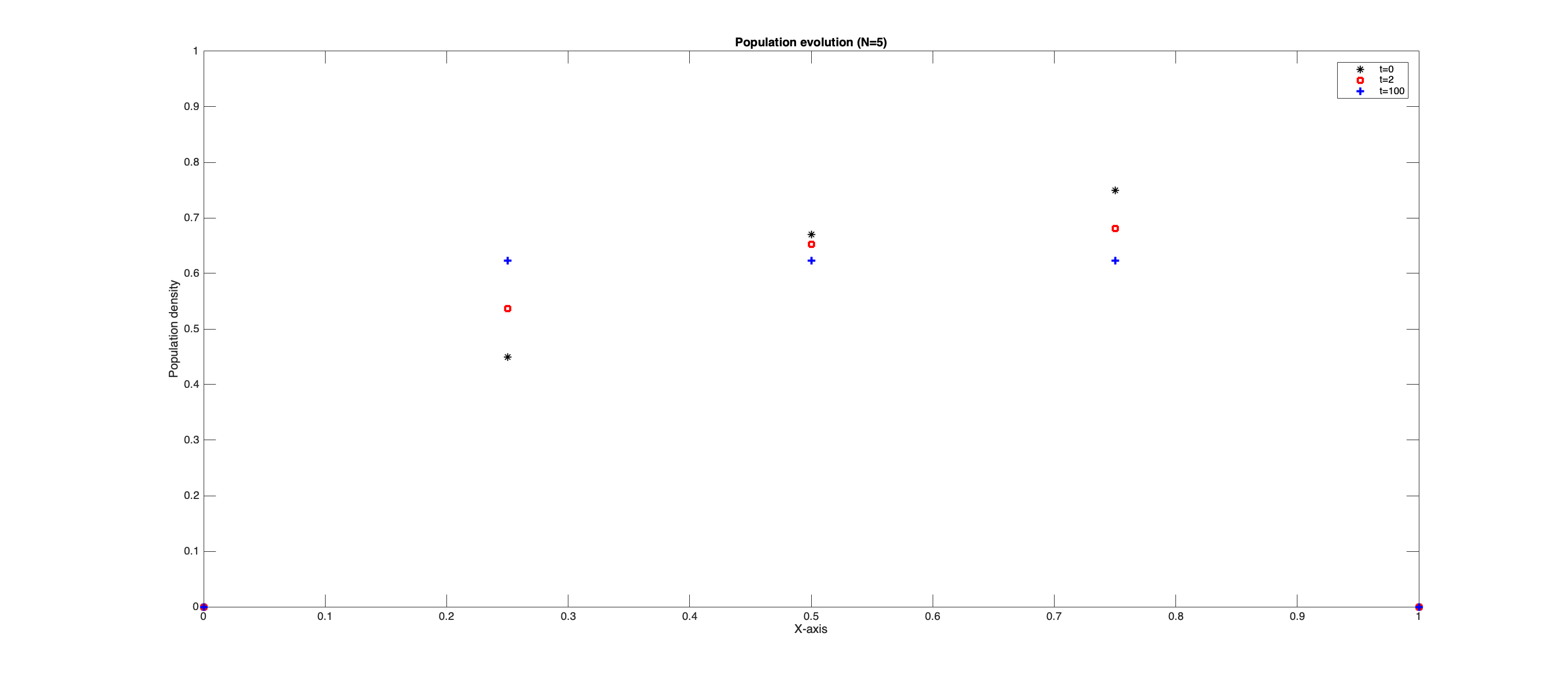}
     \caption{Initial density $u(1,0)=.45, u(2,0)=.67,u(3,0)=.75$ }\label{Fig:Data8}
   \end{minipage}\hfill
\end{figure}
\par Subcase 3: $u(3,0)<u(2,0)<u(1,0)$. We have similar results as in subcase 2 (see Figure \ref{Fig:Data9}).
\begin{figure}[!htb]
   \begin {minipage}{1.00\textwidth}
     \centering
     \includegraphics[width=1.00\linewidth]{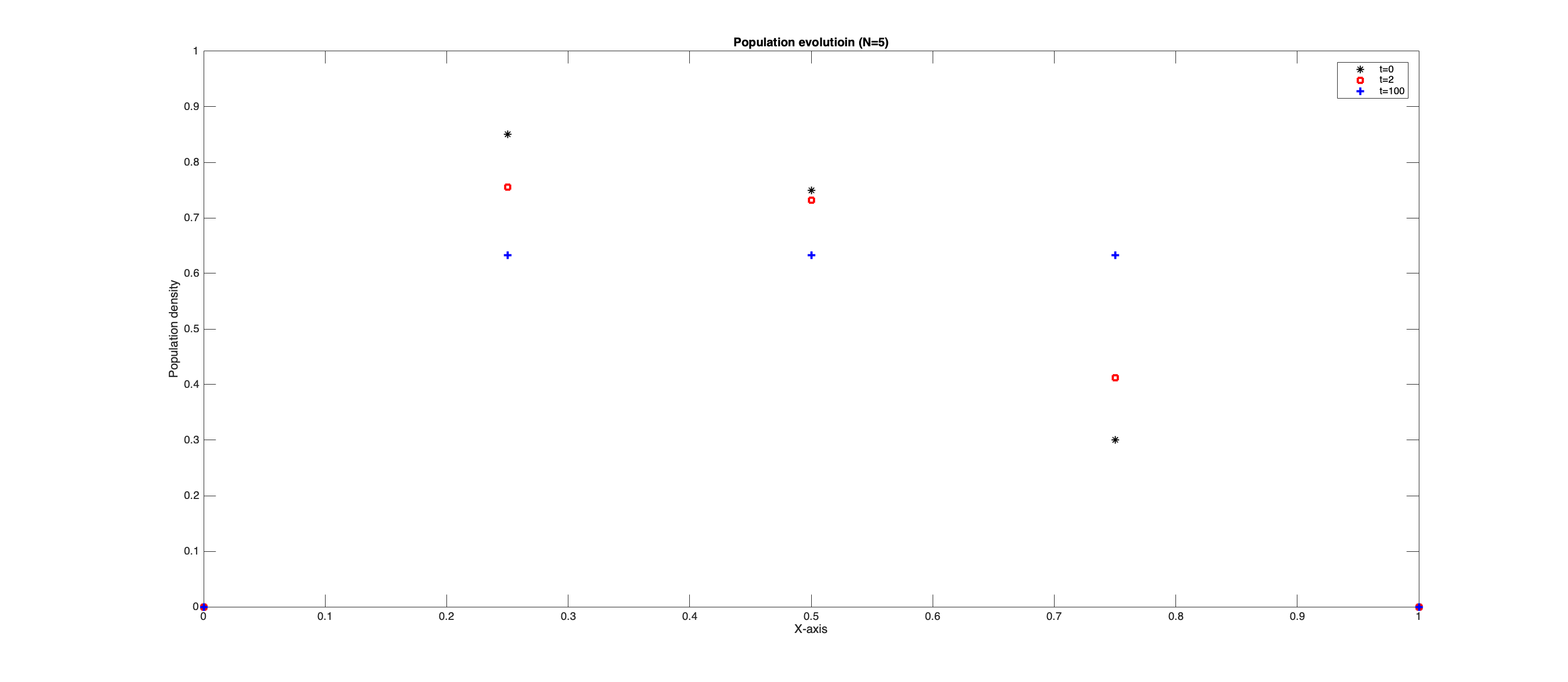}
     \caption{Initial density $u(1,0)=.85, u(2,0)=.75,u(3,0)=.3$}\label{Fig:Data9}
   \end{minipage}
\end{figure}
\par Subcase 4: $u(2,0)$ is the maximum point. From Equation \eqref{Lattice-1}-\eqref{Lattice-7}, we have $u(1,t),u(3,t)$ are increasing and $u(2,t)$ is decreasing in time first. Furthermore, we assume $u(2,t)$ is the maximum point for all time $t$, otherwise it will go to subcase 1 and subcase 2. From Equation \eqref{Lattice-3} and \eqref{Lattice-6}, $u_1^t +  u_2^t$ and $u_2^t +  u_3^t$ are decreasing. Then if $u_1^t +  u_2^t > 1$ and $u_2^t +  u_3^t > 1$ for all time $t$, from Equation \eqref{Lattice-1},\eqref{Lattice-2},\eqref{Lattice-5} $u(1,t), u(3,t)$ are increasing and $u(2,t)$ is decreasing in time $t$. We have the convergence results (see Figure \ref{Fig:Data10})
\begin{equation*}
\lim_{t\rightarrow\infty}u_1^t=\lim_{t\rightarrow\infty}u_2^t=\lim_{t\rightarrow\infty}u_3^t = \frac{1}{3}[u(1,0) + u(2,0) +u(3,0)].
\end{equation*}
If $u_1^t +  u_2^t < 1$ and $u_2^t +  u_3^t > 1$  at some point $t^*$ (similar idea for $u_1^t +  u_2^t > 1$ and $u_2^t +  u_3^t < 1$, then we go to the subcase 1 of case 3.
\begin{figure}[!htb]
   \begin{minipage}{1.00\textwidth}
     \centering
     \includegraphics[width=1.00\linewidth]{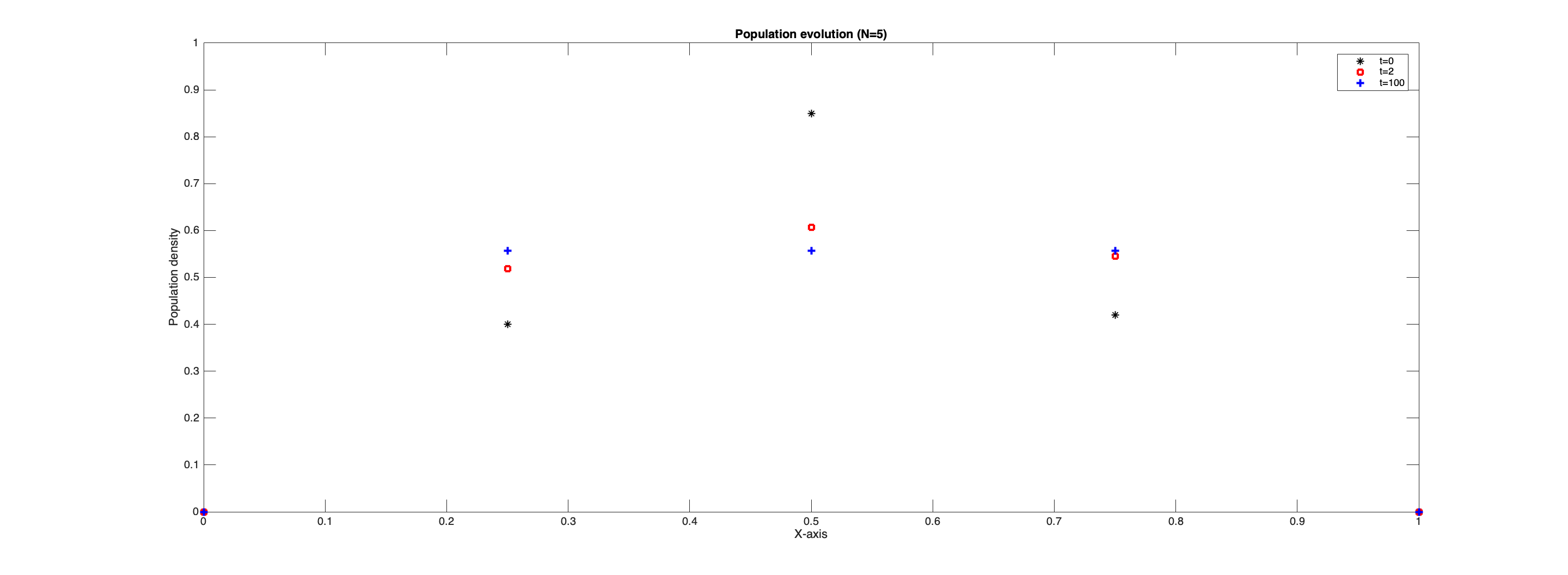}
     \caption{Initial density $u(1,0)=.40, u(2,0)=.85,u(3,0)=.42$ }\label{Fig:Data10}
   \end{minipage}\hfill
\end{figure}

%

\par {\bf Case 3}: $u(1,0)+u(2,0)<1, u(2,0)+u(3,0)>1$,
\par Subcase 1: $u(2,0) \geq  u(3,0)$. In this case $u(2,0)$ is the maximum point. From Equation \eqref{Lattice-1}-\eqref{Lattice-7}, we have $u(3,t), u(2,t) + u(3,t)$ are increasing and $u(1,t), u(1,t) + u(2,t)$ are decreasing in time $t$ first. If $u(2,t)\geq u(3,t) \geq u(1,t)$ for all time $t$, then by the bounded-ness of $u(1,t),u(2,t)$ and $u(3,t)$, we have the asymptotic convergence results (see Figure \ref{Fig:Data12}):
\begin{eqnarray*}
\lim_{t\rightarrow\infty}u_2^t&=&\lim_{t\rightarrow\infty}u_3^t=\lim_{t\rightarrow\infty}u_3^t = \frac{1}{2}[u(1,0) + u(2,0) +u(3,0)],
\\
\lim_{t\rightarrow\infty}u_1^t&=&0.
\end{eqnarray*}
If $u(3,t) \geq u(2,t) \geq u(1,t)$ at some time point $t^*$, we must have $u(2,t), u(3,t)$ in the domain $[1/2,1]$ for all time $t\geq t^*$  (the expansion of diffusion domain see \cite[Theorem 3.1]{Bao2020continuum}) and $u(1,t^*) +u(2,t^*) < 1$ and $u(2,t^*) +u(3,t^*) > 1$. The system \eqref{lattice-eq} has the asymptotic behaviors as discussed in the following subcase 2 (see Figure \ref{Fig:Data13}).
\begin{figure}[!htb]
   \begin{minipage}{1.00\textwidth}
     \centering
     \includegraphics[width=1.00\linewidth]{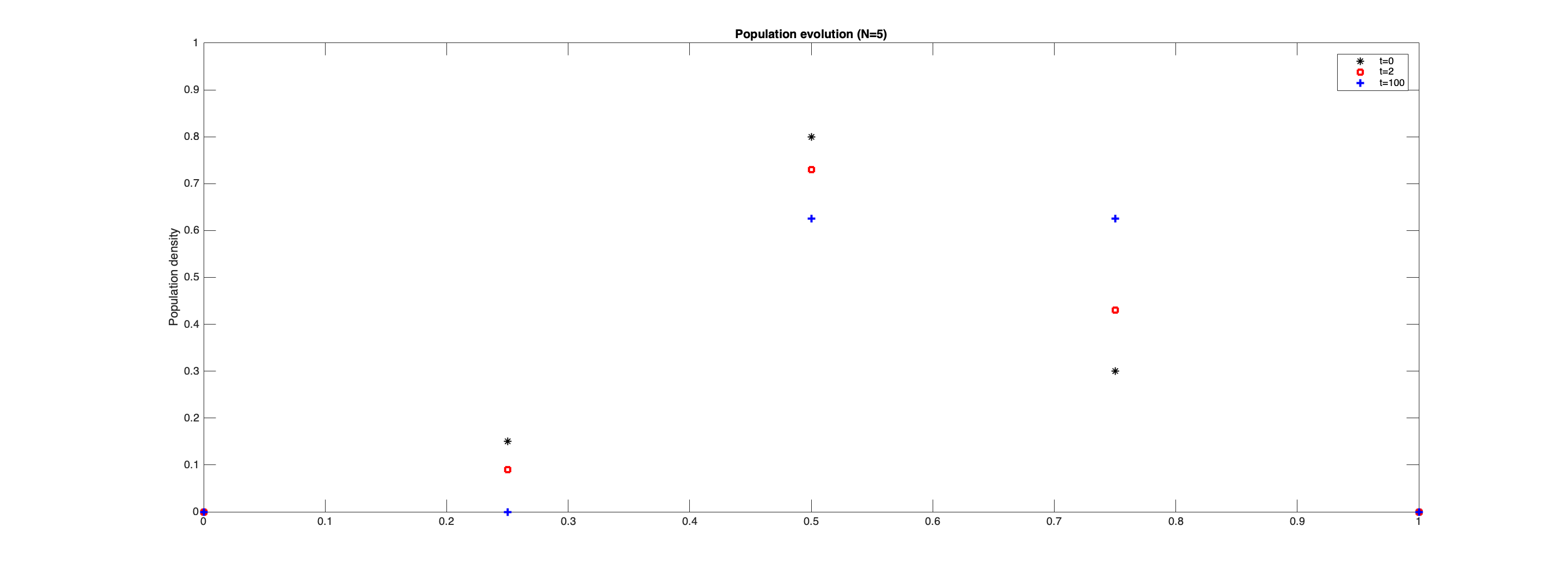}
     \caption{Initial density $u_2^0=.15, u_3^0=.80,u_4^0=.30$ }\label{Fig:Data12}
   \end{minipage}\hfill
   \begin {minipage}{1.00\textwidth}
     \centering
     \includegraphics[width=1.00\linewidth]{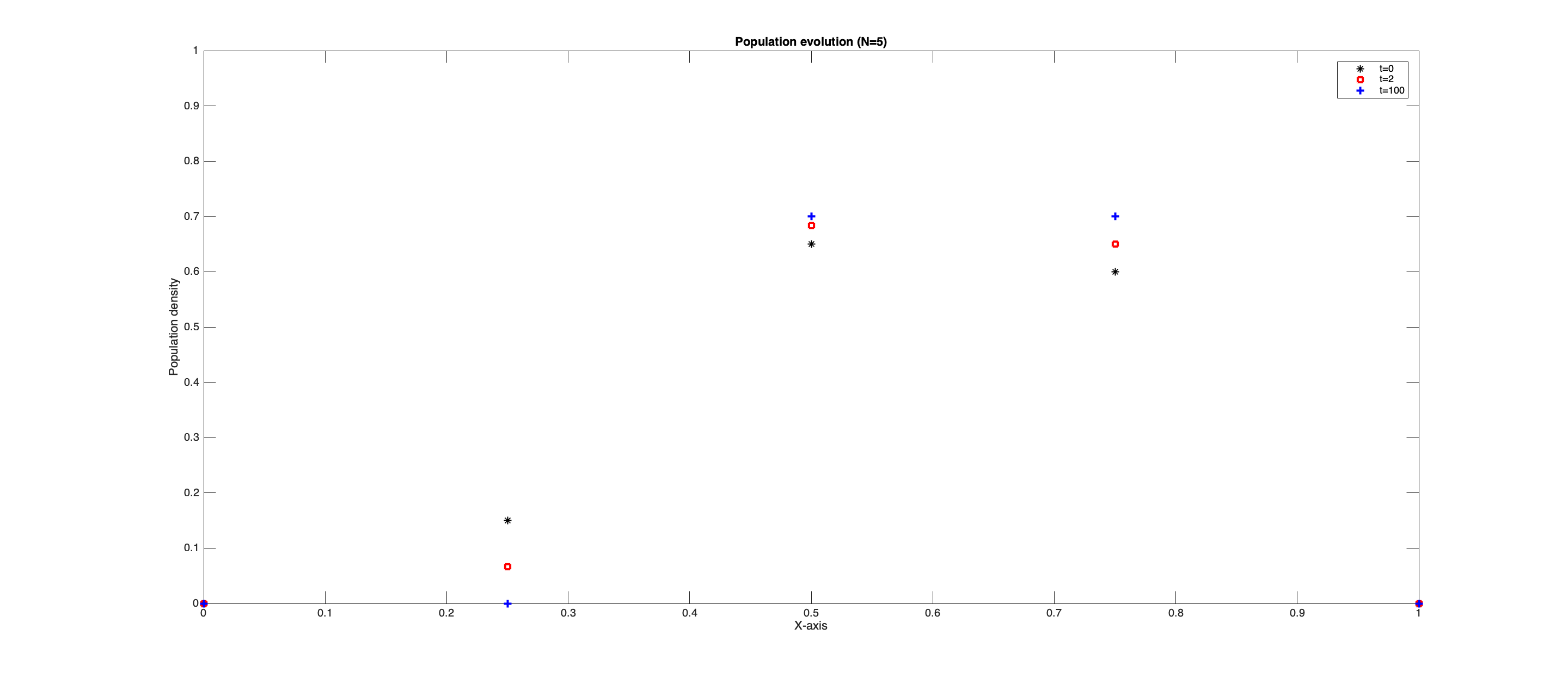}
     \caption{Initial density $u_2^0=.15, u_3^0=.65,u_4^0=.60$}\label{Fig:Data13}
   \end{minipage}
\end{figure}
\par Subcase 2:  $u(3,0) \geq  u(2,0)>1/2$. In this case we have $u(2,t)>1/2, u(3,t)>1/2$ for all $t\geq 0$ (see \cite[Theorem 3.1]{Bao2020continuum}). From equations \eqref{Lattice-1}-\eqref{Lattice-7}, we can see $u(1,t) + u(2,t)$ increase first and if $u(1,t) + u(2,t)\leq 1$ for all $t\geq 0$, then we have $u(1,t)$ is a decreasing function of time $t$ and the convergence result (see Figure \ref{Fig:Data14})
\begin{equation*}
 \lim_{t\rightarrow\infty}u_1^t=0,
\end{equation*}
\begin{equation*}
\lim_{t\rightarrow\infty}u_2^t=\lim_{t\rightarrow\infty}u_3^t = \frac{1}{2}[u(1,0) + u(2,0) +u(3,0)].
\end{equation*}
In the case when $u(1,t) + u(2,t) >1$ at some time $t^*$, then it become the case $u(1,t^*) + u(2,t^*) >1, u(2,t^*) + u(3,t^*) >1$. If $u(1,t^*) < 1/2$ and $u(3,t^*)>u(2,t^*)>1/2$, from Equation \eqref{Lattice-7} and \eqref{Lattice-4}, we have $u(1,t)\leq u(2,t) \leq u(3,t)$ for all time $t$, then $u(1,t)$ is increasing in time $t$, and we have the asymptotic result (see Figure \ref{Fig:Data15}):
\begin{equation*}
\lim_{t\rightarrow\infty}u_1^t=\lim_{t\rightarrow\infty}u_2^t=\lim_{t\rightarrow\infty}u_3^t = \frac{1}{3}[u(1,0) + u(2,0) +u(3,0)].
\end{equation*}
In the case $u(1,t^*) < 1/2$ and $1/2<u(3,t^*)<u(2,t^*)$, it goes back to the subcase 4 of case 2 and we have the convergent asymptotic results.
\begin{figure}[!htb]
   \begin{minipage}{1.00\textwidth}
     \centering
     \includegraphics[width=1.00\linewidth]{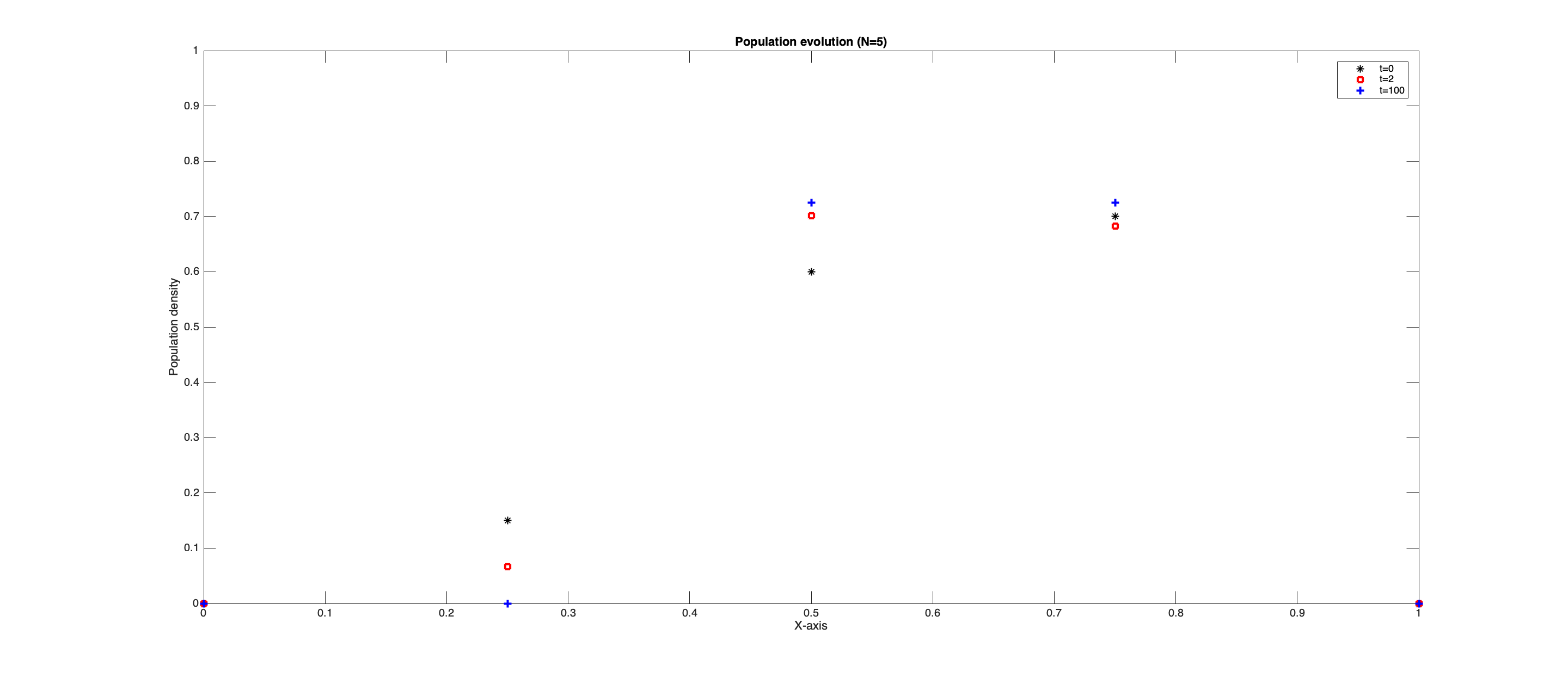}
     \caption{Initial density $u_2^0=.15, u_3^0=.60,u_4^0=.70$ }\label{Fig:Data14}
   \end{minipage}\hfill
   \begin {minipage}{1.00\textwidth}
     \centering
     \includegraphics[width=1.00\linewidth]{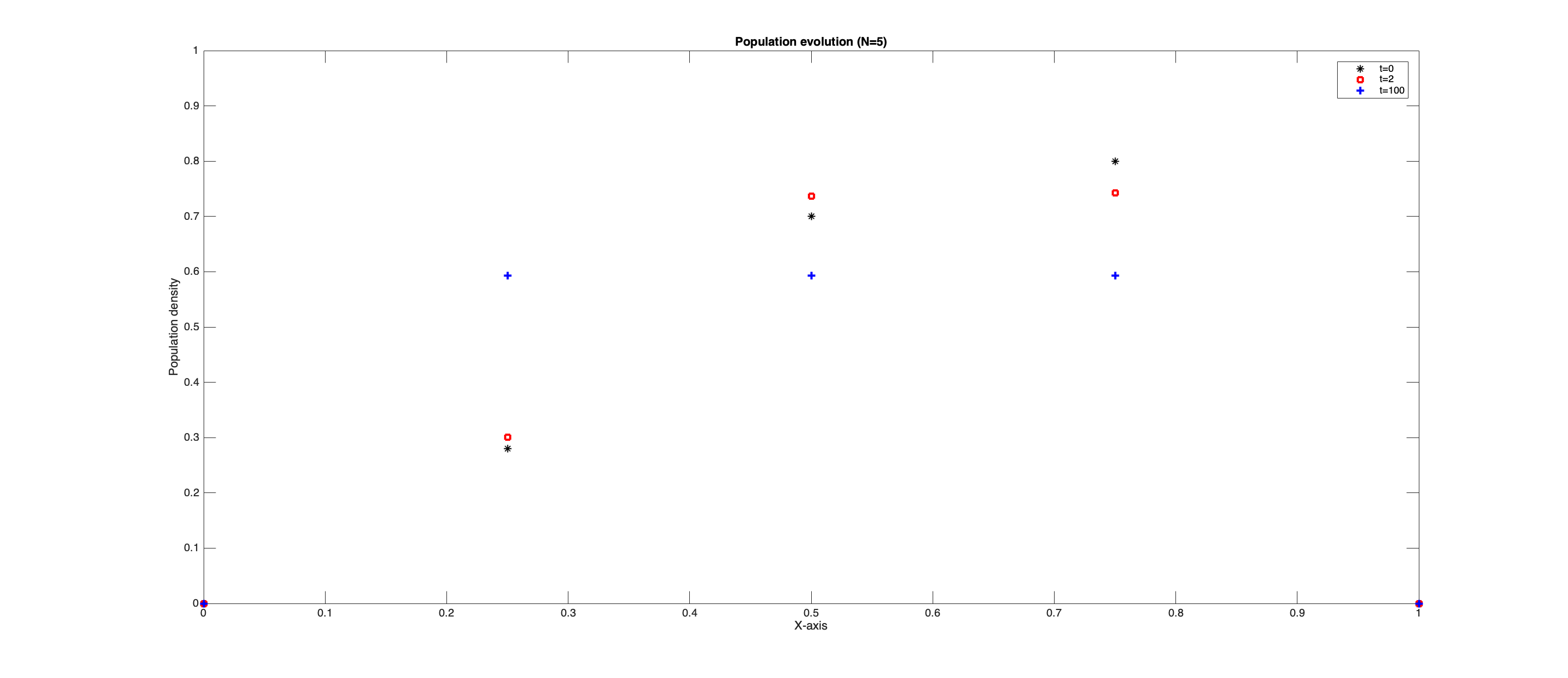}
     \caption{Initial density $u_2^0=.28, u_3^0=.70,u_4^0=.80$}\label{Fig:Data15}
   \end{minipage}
\end{figure}

\par Subcase 3: $u(3,0) >1/2>  u(2,0) >u(1,0)$. $u(2,t)$ is increasing function of time $t$ and $u(1,t)$ is a decreasing function of time $t$. If $u(2,t) <1/2$ for all time $t>0$, then from Equations \eqref{Lattice-1}-\eqref{Lattice-7}, that $u(1,t)$ is a bounded and decreasing function and it should have convergent result:
\begin{equation*}
\lim_{t\rightarrow\infty}u_1^t= 0, \lim_{t\rightarrow\infty}u_2^t=\lim_{t\rightarrow\infty}u_3^t = \frac{1}{2}[u(1,0) + u(2,0) +u(3,0)] >1/2
\end{equation*}
which is conflict to the assumption that $u(2,t) <1/2$ for all time $t$.

As a increasing function of time $t$, $u(2,t) >1/2$ after some time $t^*$, then it become the subcase 1 or 2 and we have the convergence result (see Figure \ref{Fig:Data16})
\begin{figure}[!htb]
    \begin {minipage}{0.85\textwidth}
     \centering
     \includegraphics[width=.90\linewidth]{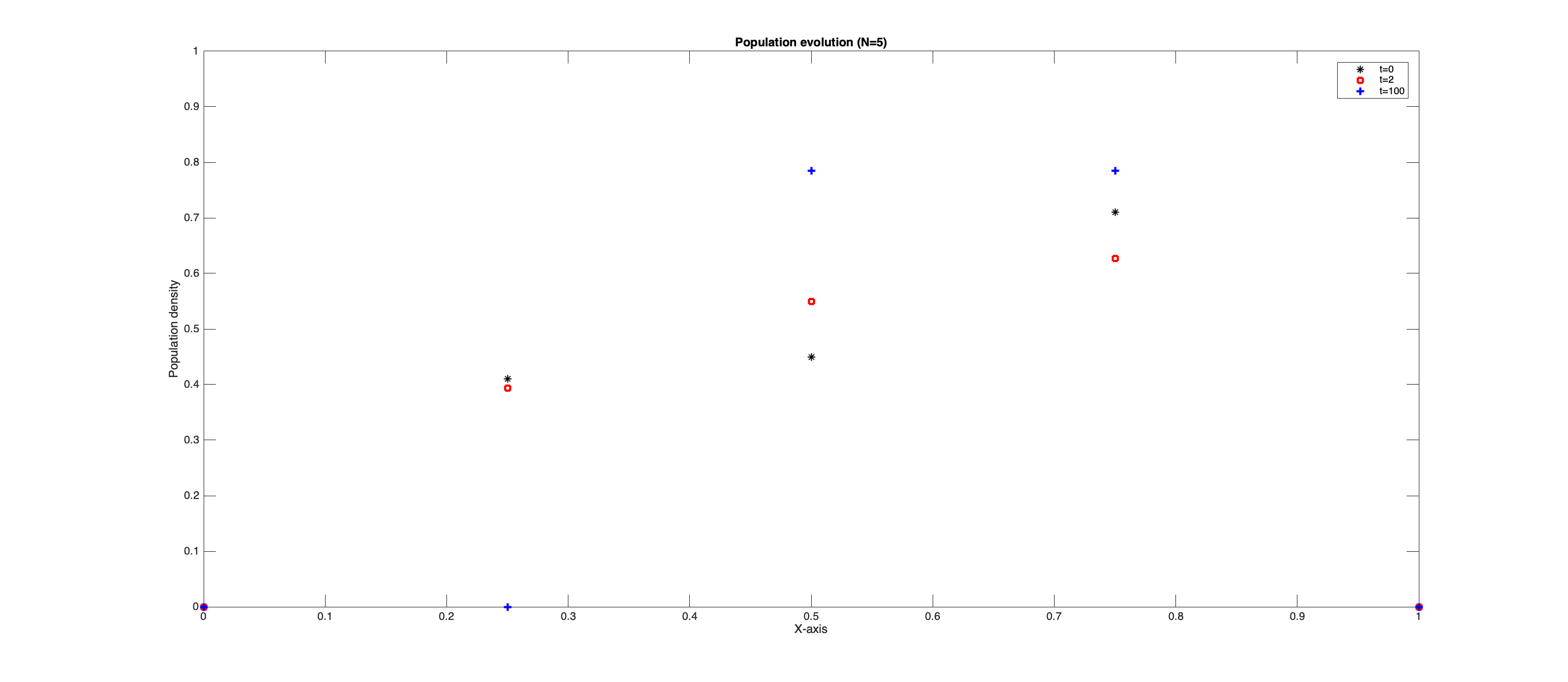}
     \caption{Initial density $u_2^0=.41, u_3^0=.45,u_4^0=.71$}\label{Fig:Data16}
   \end{minipage}
\end{figure}
\begin{rk}
When $u(1,0)+u(2,0)>1, u(2,0)+u(3,0)<1$, the asymptotic convergence results should be similar to Case 3 and we omit details here.
\end{rk}

\par {\bf Case 5}: $u(1,0)+u(2,0)=1, u(2,0)=u(3,0)$, by using equations \eqref{Lattice-1}-\eqref{Lattice-6}, we can see that this is an trivial steady state solution.

\section{Properties of the Continuous Model}
In this section, we first study existence, maximum principle and asymptotic behaviors of the backward-forward parabolic equation \eqref{Master-eq} with large initial solution. Then, we will further investigate non-existence result of Equation \eqref{Master-eq} with general initial solution which is similar to other backward-forward parabolic equation (see \cite{Smarrazzo2012degenerate} and more).

Before investigating the existence, maximum principle of backward-forward parabolic equation \eqref{Master-eq}, we first introduce the definition of weak solution of Equation \eqref{Master-eq} with condition \eqref{Diffusion-cond}.
\begin{defn}[Weak solution]\label{Def1}
 A locally continuous function $u(x,t)$ is said to be a weak solution of the backward-forward equation \eqref{Master-eq} with condition \eqref{Diffusion-cond} if
 \begin{equation}\label{weak-solu-cond1}
 \int_0^1[u^2+|D(u)|u_x^2]dx
 \end{equation}
 is uniformly bounded for $t\in [0,T]$ and for any test function $\phi(x,t)$ in $C^1_0[Q_T]$,
 \begin{equation}\label{weak-solu-cond2}
  \int_0^T\int_0^1[u\phi_t-D(u)u_x\phi_x]dxdt=0.
 \end{equation}
\end{defn}
\begin{rk}
In the special case when the diffusion coefficient $D(u)$ is a constant $C$, Equation \eqref{Master-eq} becomes the classical parabolic equation.
\end{rk}
\begin{rk}
The boundary value of $D(u)u_x$ can be zero or nonzero constant as $x$ approaches to the boundary in Equation \eqref{Master-eq} with large initial solution ($u_0 > \alpha$) on $(0,1)$ and Dirichlet boundary condition. Here we borrow the ideas of boundary degeneracy in \cite{Cannarsa2008carleman} and restrict us to the condition \eqref{weak-solu-cond1} in the weak solution of Dirichlet or Neumann boundary.
\end{rk}

When the initial solution is large enough, we have the following existence and maximum principle results.

\begin{thm}\label{Main-thm1}
 Suppose that $D(u)$ satisfies \eqref{Diffusion-cond}, $u(x,0)\in C^{1,\beta}([0,1]),\beta\in(0,1)$, and $\alpha<u(x,0)\leq 1$. Then there exists a classical solution in $C^{2,1}(Q_T)$ of the equation \eqref{Master-eq} for all $T>0$ with boundary conditions $D(u)u_x(0,t)=D(u)u_x(1,t)=0$. Furthermore, $\alpha<u(x,t)\leq 1$, for all $0\leq t\leq T$.
\end{thm}
\begin{proof}
First, because of the locally continuity of the initial solution $u(x,0)$, we have $D(u) >0$ for small $T_1>0$ and the existence of the classical $C^{2,1}(Q_{T_1})$ solution is guaranteed by \cite{Lieberman1996second}, \cite{OVN}.

Next, we prove the following claim.

{\bf Claim.} For all $T_1 \geq t\geq 0$, we have the following maximum principle
\begin{equation}\label{Max-p}
 \alpha<\min_{x\in[0,1]}u(x,0) \leq u(x,t) \leq \max_{x\in[0,1]} u(x,0)\leq 1.
\end{equation}
We first prove the lower bound estimate. Suppose the lower bound estimate in \eqref{Max-p} does not exist, then we have
\begin{equation}\label{Min-value}
 \min_{x\in[0,1],0\leq t\leq T_1}u(x,t)=u(x_0,t_0),\quad T_1\geq t_0>0.
\end{equation}
Let $\nu(x,t) = u(x,t)+\epsilon t$ with $\epsilon>0$, then
\begin{eqnarray}\label{Min-eq}
 \nu_t=u_t+\epsilon &=&[D(u)u_x]_x+\epsilon
 =[D(u)\nu_x]_x+\epsilon,\nonumber
 \\
 &=& D'(u)\nu_x^2+D(u)\nu_{xx}+\epsilon,
\end{eqnarray}
$\nu(x,t)$ can not have a minimum for $t>0$. Otherwise $\nu(x_1,t_0), x_1\in[0,1] $ is the mimimum, then $\nu_x(x_1,t_0)=0,\nu_{xx}(x_1,t_0) \geq 0,$  which is a contradiction to Equation \eqref{Min-eq}.
So we obtain
\begin{equation*}
\min_{x\in[0,1],0\leq t\leq T_1}\nu(x,t)=\min_{x\in[0,1]}\nu(x,0)=\min_{x\in[0,1]}u(x,0),
\end{equation*}
 which means
\begin{equation}\label{eq:012}
 u(x,t)+\epsilon t\geq \min_{x\in[0,1]}u(x,0),\quad \forall \epsilon>0.
\end{equation}
For the arbitrariness of $\epsilon$, and let $\epsilon\rightarrow 0,$ we obtain $u(x,t)\geq\min_{x\in[0,1]}u(x,0)$,
which leads to  $u(x,t)\geq \alpha+\delta_1$ for some small number $\delta_1$ and $t>0$.

Furthermore, we see this classical solution of Equation \eqref{Master-eq} satisfies the following $L^p-$estimate:
\begin{eqnarray*}
\frac{d}{dt}\int_0^1|u|^pdx &=& p\int_0^1 u_tu|u|^{p-2}dx = p\int_0^1 u|u|^{p-2}(D(u)u_x)_xdx
\\
&=&-p(p-1)\int_0^1D(u)|u|^{p-2}|u_x|^2 dx\leq -cp(p-1)\int_0^1|u|^{p-2}|u_x|^2dx \leq 0.
\end{eqnarray*}
It follows from these inequalities that the following $L^p-$ norm can be bounded by the initial data:
\begin{equation}\label{Up-bound}
 (\int_0^1|u(x,t)|^p)^{\frac{1}{p}} \leq (\int_0^1|u(x,0)|^p)^{\frac{1}{p}}, \quad \forall p\in [2,\infty).
\end{equation}
Thus, as $p\to\infty$, we see that $\|u(\cdot,t)\|_{L^\infty([0,1])} \leq |u(\cdot,0)\|_{L^\infty([0,1])} \leq 1$ $\forall T_1 \geq t>0$. Because the lower and upper bound in \eqref{Max-p} are independent of $T_1$, we can iterate the above process to guarantee inequalities in \eqref{Max-p} for all $0<T <\infty$ and obtain the results in the statement.
\end{proof}
\begin{rk}
Under the assumption $u(x,0)\in C^{1,\beta}([0,1]),\beta\in(0,1), \alpha<u(x,0)\leq 1$ the solution of \eqref{Master-eq} actually belong to $C^\infty(Q_T)$ \cite{Bao2020continuum}, and this will play as a key method in proving non-existence result of Equation \eqref{Master-eq} with a general initial solution.
\end{rk}

Equation \eqref{Master-eq} with uniform parabolic diffusion coefficient $D(u)\geq c>0$ and its corresponding discrete equation
have the vanishing asymptotic results with Dirichlet boundary, but the aggregation diffusion lattice equation \eqref{lattice-eq} has a rich asymptotic dynamical behaviors (see section 3).

For the comparison reason, we only focus on 1-dimensional case:
\begin{equation}\label{uniform-parab}
\begin{cases}
\frac{\partial u}{\partial t} = \frac{\partial}{\partial x}(D(u)\frac{\partial u}{\partial x}), \quad (x,t)\in [0,1]\times [0,T],
\\
u(0,t)=u(1,t)=0.
\end{cases}
\end{equation}
The asymptotic behavior extension of Equation \eqref{uniform-parab} to the high dimension will be the same as in 1-dimensional case.

\begin{thm}\label{Main-thm2}
Suppose $u_0\in C^{1,\beta}[0,1]$ and  $1\geq u_0\geq 0$, then the solution $u(x,t)$ of System \eqref{uniform-parab} with uniform parabolic diffusion coefficient $D(u)\geq c>0$ goes to zero as $t\to\infty$.
\end{thm}
\begin{proof}
By the classical parabolic equation existence theory \cite{OVN}, System \eqref{uniform-parab} has a classical solution with $u_0\in C^{1,\beta}[0,1]$.

Now we consider the time derivative of the total population density
\begin{eqnarray}
 \frac{1}{2}\frac{d}{dt}\int_0^1 u^2dx &=& \int_0^1 uu_tdx =\int_0^1 u(D(u)u_x)_x dx\nonumber
  \\
  &=& \int_0^1 -D(u)u_x^2dx \leq -c\int_0^2 u_x^2dx\nonumber
  \\
  &\leq & -c\delta \int_0^1 u^2dx \quad (\mbox{Poincar$\acute{e}$ inequality})
\end{eqnarray}
Then by Gronwall's inequality, we have
\begin{equation}
\int_0^1 u^2(x,t) dx \leq (\int_0^1 u(x,0)^2dx) e^{-2c\delta t}.
\end{equation}
Let $t\to +\infty$, we have
\begin{equation}
\lim_{t\to +\infty}\int_0^1 u^2(x,t)dx = 0,
\end{equation}
which means the total population density will vanish as time continuous and we have a similar result in its corresponding discrete system in section 3 for classical discrete heat equation. The biological explanation is that the species will extinct when it has pure diffusion migration.
\end{proof}

When the initial solution is large enough, Equation \eqref{Master-eq} has the following asymptotic behavior with Neumann boundary \cite{Bao2020continuum}.
\begin{thm}[Theorem 3.2 \cite{Bao2020continuum}]\label{thm2}
 Suppose $u_0\in C^{1,\beta}[0,1], \alpha < u_0 \leq 1$, and $u(x,t)$ is the solution of the equation \eqref{continu-eq} with Neumann boundary condition, then solution $u(x,t)$ will go to the constant $C=\int_0^1u(x,0)dx$.
\end{thm}

\begin{rk}
In the uniformly parabolic case where $D(u)\geq c>0$ and $u_0\geq 0, u_0\in C^{1,\beta}[0,1]$, by using the same arguments in the proof of \ref{thm2}, the solution has the same asymptotic convergence result.
\end{rk}

Because of the aggregation or reverse diffusion, by using Theorem \ref{Main-thm1}, we have the following non-existence result of system \eqref{continu-eq} with Dirichlet boundary.
\begin{thm}\label{Main-thm3}
Suppose $u(x,0)\in C^{1,\beta}[0,1], 0<\beta<1$, then Equation \eqref{Master-eq} with diffusion condition \eqref{Diffusion-cond} has no weak solution with Dirichlet boundary condition.
\end{thm}
\par
\begin{proof} Suppose Equation \eqref{Master-eq} with diffusion condition \eqref{Diffusion-cond} has a weak solution, by the linear transform $\tau=T-t$, we reverse the time interval and obtain
\begin{equation}\label{Rev-eq}
 u_\tau=-(D(u)u_x)_x,\quad \forall (x,t)\in Q_T.
\end{equation}
Because $u(0,t)=u(1,t)=0$ and by the local continuity, assuming $T$ is small enough, Equation \eqref{Rev-eq} has a local solution on the domain $[0,\epsilon_1]\times [0,T_{\epsilon_1}]$. Then by Theorem \ref{Main-thm1}, $u(x,t) \in C^{2,1}((0,\epsilon_1]\times [0,T_{\epsilon_1}))$, especially, $u(x,0)\in C^{2,1}(0,1]$ which is more regular than $C^{1,\beta}[0,1]$, and we have the contradiction.
\end{proof}


\section{Conclusion}
In this paper, we established a backward-forward parabolic equation from an individual-based model to describe the aggregation and diffusion in cell evolution. When the initial density is large enough, the existence, maximum principle and the asymptotic behaviors of the continuous Equation \eqref{continu-eq} are investigated with Neumann boundary. Also the non-existence of solution of Equation \eqref{continu-eq} is established with Dirichlet boundary condition.

In the corresponding discrete model with Dirichlet boundary condition, species eventually vanish with only diffusion mechanism. However, with the help of aggregation when cell density is small, density conservation is guaranteed which validate our observation in biology that aggregation help the survive of species under dangerous environment. The species converges to the steady-state when the initial solution is in the diffusion domain or large enough. However, different patterns emerges with general initial solution in the special 5 points lattice model with Dirichlet boundary which give us a possible explanation that patterns in biology are created by the interaction between aggregation and diffusion.

\end{document}